\documentclass[abstracton, paper=a4, fontsize=11pt,DIV=12,bibliography=totoc]{scrartcl}

\usepackage{geometry}
\usepackage[ngerman,english]{babel}
\usepackage{amsmath}
\usepackage{amsrefs}
\usepackage{amssymb}
\usepackage{amsthm}
\usepackage{caption,subcaption}
\usepackage{cite}
\usepackage{color}
\usepackage{enumitem}
\usepackage{framed}
\usepackage{graphicx}
\usepackage{hyperref}
\usepackage{layout}
\usepackage{mathtools}
\usepackage{mathrsfs}
\usepackage{placeins}
\usepackage{pdfpages}
\usepackage{tikz}
\usepackage{txfonts}
\usepackage[normalem]{ulem}
\usepackage{aligned-overset}
\usepackage{enumitem}
\usepackage{mathrsfs}
\usepackage{enumitem}
\usepackage{wrapfig}

\newtheorem{theorem}{Theorem}

\newtheorem{corollary}[theorem]{Corollary}
\newtheorem{definition}[theorem]{Definition}

\newtheorem{lemma}[theorem]{Lemma}
\newtheorem{notation}[theorem]{Notation}
\newtheorem{proposition}[theorem]{Proposition}
\newtheorem{remark}[theorem]{Remark}

\newcommand{\stkout}[1]{\ifmmode\text{\sout{\ensuremath{#1}}}\else\sout{#1}\fi}

\newcommand{\fr}[2]{\frac{{#1}}{{#2}}}

\newcommand{\babla}{\boldsymbol{\nabla}}
\newcommand{\bdx}{\,\mathrm{d}\mathbf{x}}
\newcommand{\bx}{\mathbf{x}}
\newcommand{\belta}{\boldsymbol{\Delta}}
\DeclareMathOperator{\Div}{div}
\DeclareMathOperator{\Id}{Id}
\DeclareMathOperator{\loc}{loc}
\newcommand{\one}{\mathbf{1}}

\newcommand{\dx}{\,\mathrm{d}x}
\newcommand{\dz}{\,\mathrm{d}z}
\newcommand{\dk}{\,\mathrm{d}k}
\newcommand{\dS}{\,\mathrm{d}S}
\newcommand{\dd}{\mathrm{d}}

\newcommand{\ls}{\lesssim}
\newcommand{\gs}{\gtrsim}
\newcommand{\bB}{\mathbf{B}}

\newcommand{\dt}{\, \mathrm{d}t}

\newcommand{\ddt}{\frac{{\mathrm d}}{{\mathrm dt}}}

\newcommand{\eps}{\varepsilon}

\newcommand{\R}{\mathbb{R}}
\newcommand{\N}{\mathbb{N}}

\newcommand{\sppt}{{\mathrm{supp}}}

\newcommand{\mres}{\mathbin{\vrule height 1.6ex depth 0pt width
0.13ex\vrule height 0.13ex depth 0pt width 1.3ex}}
\renewcommand{\llcorner}{\mres}

\DeclarePairedDelimiter\norm{\lVert}{\rVert}
\newcommand\abs[1]{\left| #1\right|}
\newcommand\set[1]{\left\{#1\right\}}
\newcommand\bra[1]{\left(#1\right)}
\newcommand\jump[1]{\left[#1\right]}

\newcommand{\dissipation}{\ensuremath{{D}}}
\newcommand{\D}{\dissipation}

\newcommand{\E}{\mathcal{E}}
\newcommand{\F}{\mathcal{F}}
\newcommand{\Hdist}{\mathcal{H}}
\newcommand{\Hd}{\dot{H}}

\newcommand{\uu}{u}
\newcommand{\ub}{\bar{\uu}}
\newcommand{\hh}{h}
\newcommand{\hb}{\bar{\hh}}
\newcommand{\V}{\mathcal{V}}
\newcommand{\Vb}{\overline{\mathcal{V}}}
\newcommand{\Vt}{\widetilde{\mathcal{V}}}
\newcommand{\vel}{V}
\newcommand{\vt}{\tilde{v}}
\newcommand{\delV}{\mathit{\Delta}V}
\newcommand{\delchi}{\mathit{\Delta}\chi}
\newcommand{\Po}{\mathcal{P}}
\definecolor{darkred}{rgb}{0.9,0.1,0.1}

\numberwithin{equation}{section}   
\numberwithin{theorem}{section}

\mathtoolsset{showonlyrefs}

\begin{document}

\title{
	Convergence to the planar interface for a nonlocal free-boundary evolution}
\author{
Felix Otto\begin{footnote}{MPI for Mathematics in the Sciences, E-mail address: \href{mailto:felix.otto@mis.mpg.de}{felix.otto@mis.mpg.de}.}\end{footnote},
Richard Schubert\begin{footnote}{University of Bonn, E-mail address: \href{mailto:schubert@iam.uni-bonn.de}{schubert@iam.uni-bonn.de}.}\end{footnote},
 and
	Maria G. Westdickenberg\begin{footnote}{RWTH Aachen University, E-mail address: \href{mailto:maria@math1.rwth-aachen.de}{maria@math1.rwth-aachen.de}.}\end{footnote}
}
\date{\today}
\maketitle
\begin{abstract}
We capture optimal decay for the Mullins-Sekerka evolution, a nonlocal, parabolic free boundary problem from materials science. Our main result establishes convergence of BV solutions to the planar profile in the physically relevant case of ambient space dimension three. Far from assuming small or well-prepared initial data, we allow for initial configurations that do not have graph structure and in which the region of positive phase is not simply connected, hence explicitly including the regime of Ostwald ripening. In terms only of initially finite (not small) excess mass and excess surface energy, we establish that the surface becomes a Lipschitz graph within a fixed timescale (quantitatively estimated) and remains trapped within this setting for future times. To obtain the graph structure, we leverage regularity results from geometric measure theory. At the same time, we extend a duality method previously employed for one dimensional PDE problems to higher dimensional, nonlocal geometric evolutions. Optimal algebraic decay rates of excess energy, dissipation, and graph height are obtained.


\end{abstract}

\section{Introduction}

The Mullins-Sekerka interfacial evolution (abbreviated MS in the following) is  a nonlocal free boundary
evolution: The normal velocity of the interface is a nonlocal function of the interface and, in particular, its mean curvature.
Essential features of MS are the preservation of (signed) mass and the reduction
of surface area. In addition, there is a scale invariance: The solution space is invariant under
a rescaling of
length by a factor of $\lambda$ and time by a factor of $\lambda^3$. Hence MS is a geometric version of a third-order parabolic equation.

We are interested in the relaxation to a planar interface for fairly wild initial data: The initial interface is required neither to be a graph nor to consist of a single connected component. Singularities are not a matter of idle curiosity in MS but rather a fact of life:
A well-known and naturally occurring configuration is that of \emph{Ostwald ripening}, in which the positive phase is distributed in multiple, disconnected islands, and the smaller islands shrink and disappear (a singularity in the flow). We allow for such configurations.

We do not study existence but rather analyze the qualitative and quantitative properties of any BV solution. We will obtain optimal relaxation rates to the flat limiting configuration based only on monitoring the excess (absolute) mass or ``volume'' (between the positive phase $\Omega_+\subseteq \R^{d+1}$ and  $\{z>0\}$):
\begin{align}
  \V:=\int_{\R^{d+1}} |\chi|\,\bdx,\qquad \text{where}\qquad
  \chi:=\mathbf{1}_{\Omega_+}-\mathbf{1}_{\R^d\times\{z>0\}}\label{eq:vfinite}
\end{align}
and the
excess surface area (compared to the flat configuration):
\begin{align}
  \E&=\int_{\Gamma}1-e_z\cdot n\dS,\label{eq:E}
\end{align}
where $e_z$ is the unit vector in $z$-direction and $n$ is the \textbf{unit inward normal} to the boundary of the positive phase $\Gamma=\partial \Omega_+$.

In its strong form, the Mullins-Sekerka dynamics consists of the evolution of a $d$-dimen\-sio\-nal hypersurface $\Gamma$, which is the boundary of a region $\Omega_+$ and induces a potential $f:\R^{d+1}\to \R$ as the two-sided harmonic extension of the mean curvature (the sum of the principal curvatures) $H$ of $\Gamma$ according to
\begin{align}\label{eq:f}
  \begin{cases}
    \belta f=0 & \text{in }\R^{d+1}\setminus\Gamma,\\
    f=H &\text{on }\Gamma.
  \end{cases}
\end{align}
The surface $\Gamma$ itself evolves with normal velocity determined by the jump in normal derivative of $f$ across $\Gamma$ via:
\begin{align}\label{eq:Vel}
  \vel:=-[\babla f\cdot n],
\end{align}
where here and throughout we use square brackets to indicate  $[\babla f\cdot n]=\nabla f_+\cdot n -\nabla f_-\cdot n$.
\begin{definition}\label{def:weaksol}
For given initial data $\one_{\Omega_+(0)}\in BV_{\loc}(\R^{d+1};\set{0,1})$, we call an $L^1$-continuous family of characteristic functions $\one_{\Omega_+(t)}\in BV_{\loc}(\R^{d+1};\set{0,1})$ together with\\ $f\in L^2(0,\infty;\Hd^1(\R^{d+1}))$ a weak BV solution of the MS dynamics if the following conditions are satisfied.
\begin{itemize}
\item[(1)] For any $T>0$, all $\varphi\in C_c^\infty(\R^{d+1}\times [0,
  T])$, and $\chi$ as in \eqref{eq:vfinite} there holds
\begin{align}
\lefteqn{\left(\int_{\R^{d+1}}\chi\varphi\bdx\right)(T)- \left(\int_{\R^{d+1}}\chi\varphi\bdx\right)(0)}\notag\\
&=\int_0^T \int_{\R^{d+1}}\chi\partial_t\varphi\bdx \dt -\int_0^T\int_{\R^{d+1}}\babla f\cdot \babla\varphi\bdx\dt.\label{eq:weak_dyn}
\end{align}
\item [(2)] Let $\Gamma(t)\subseteq \R^{d+1}$ be the support of the Radon measure $\babla \one_{\Omega_+}$ and let $n$ be the Radon-Nikodym derivative of $\babla \one_{\Omega_+}$ with respect to its total variation measure $|\babla \one_{\Omega_+}|$. For almost all $t>0$, and all $\xi\in C_c^\infty(\R^{d+1};\R^{d+1})$, there holds
\begin{align}
  \int_{\Gamma}&\!(\Id-n\otimes n)\colon\! \nabla \xi \dS =\int_{\R^{d+1}}\!\one_{\Omega_+}\Div (f\xi)\bdx,\label{eq:weak_H}
\end{align}
where the notation $\int_{\Gamma}\dS$ is to be understood as $\int_{\R^{d+1}}\dd |\babla \one_{\Omega_+}|$.
\end{itemize}
\end{definition}
We will see in Section~\ref{S:graph} that there exists a measurable function $H$ such that the integrals in \eqref{eq:weak_H} are equal to $-\int_{\Gamma} H\xi\cdot n \dS$, and $H$ can be considered the generalized mean curvature. (Note that with our choice of orientation, the mean curvature $H$ is positive where $\Omega_+$ is convex.) We remark that \eqref{eq:weak_dyn} encodes $\partial_t \chi-\Delta f=0$ in the sense of distributions; in particular, $f$ is harmonic outside $\Gamma$. For a smooth solution, \eqref{eq:Vel}
follows from ~\eqref{eq:weak_dyn}
via integration by parts, and the boundary condition \eqref{eq:f} follows from~\eqref{eq:weak_H}, integration by parts, and the equality with $-\int_{\Gamma} H\xi\cdot n \dS$ observed above.

We will work throughout the paper with BV solutions that satisfy the hypothesis:
\begin{quote}\textbf{Hypothesis (H).}
The energy is absolutely continuous and nonincreasing with a.e.\ time derivative
\begin{align}
-\ddt\E\geq D \coloneqq   \int_{\R^{d+1}}|\babla f|^2\,\bdx,\label{eq:Edot}
\end{align}
so that for all $s<t$ there holds
\begin{align}\label{eq:D_int}
  \E(t)-\E(s)\leq - \int_s^t D(\tau)\, \dd\tau.
\end{align}
Furthermore we assume that the solution is smooth forward in time if it reaches a state with graph structure and small enough Lipschitz constant, i.e., that there exists $\varepsilon_0>0$ such that
$$\Omega_+(t)=\set{(x,z)\in \R^{d+1}:z>h(x,t)} \quad \text{and}\quad \norm{\nabla h(\cdot,t)}_\infty\le \varepsilon_0$$ implies that
the solution is smooth in space-time for some interval $[t,T)$, $T>t$.
\end{quote}
In particular, Hypothesis \textbf{(H)} allows for singular events that occur when connected components of $\Omega_+$ disappear (as in Ostwald ripening) or collide. We motivate our hypothesis in the context of the available literature in Subsection~\ref{subsec:lit} below. The reader who prefers not to consider weak solutions may alternatively suppose that there is a strong solution outside of finitely many singular times.

\begin{remark}
  The De-Giorgi type inequality from \cite{HS22} suggests replacing~\eqref{eq:Edot} by $-\ddt\E\geq \frac 12 D$, in which case our results carry through unchanged.
\end{remark}

The Mullins-Sekerka evolution is the gradient flow of the excess surface energy $\E$
with respect to the $\Hd^{-1}$ metric, where the configuration space is given by characteristic functions. We will regularly make use of the energy dissipation from Hypothesis \textbf{(H)} in the form
\begin{align}\label{eq:E_bound_simple}
\E\le \E_0\qquad\text{ for all }\,t>0.
\end{align}

We obtain optimal algebraic-in-time decay rates for the energy gap within the $L^1$ setting, i.e., for initially finite $L^1$ distance to a half-space, $\V_0<\infty$. The $L^1$ norm is naturally related to the fact that the dynamics is mass conserving, i.e., the value $\int_{\R^{d+1}}\chi \bdx$ is a conserved quantity of the evolution. (Note however that the $L^1$ norm $\V$ \emph{is not typically preserved}; an elementary example is the shrinking of two initially circular islands symmetrically located on either side of a flat interface.) Our most interesting result is for $d=2$, where we show that for rough initial data, the evolution \emph{enters the graph setting} within finite time, after which it remains trapped close to the linear evolution.

Our main result for $d=2$ (ambient space dimension three) is:

\begin{theorem}[Relaxation to flat]\label{thm:main}
Consider an initial set $\Omega_{+,0}\subseteq\R^3$ of locally finite perimeter with reduced boundary $\Gamma_0$ having finite excess mass $\V_0$ and excess energy $\E_0$. Consider a solution of MS satisfying  \textbf{(H)}. Then
for all times $t>0$, the excess mass is bounded by
\begin{align}\label{eq:V_bounded}
\V\ls \V_0+\E_0^{\frac {3}2},
\end{align}
and the energy decays according to
\begin{align}
\E&\ls \min\set{\E_0,\frac{\V_0^2+\E_0^3}{t^{\frac {4}3}}}.\label{eq:E_decay}
\end{align}
Moreover,
there exists a time $T_g\sim \E_0^{3/2}$ (``graph-time'') such that
for all $t\geq T_g$,
the interface $\Gamma(t)$ is the graph of a Lipschitz function
\begin{align}\label{eq:graph}
\Gamma(t)=\set{(x,z)\in \R^{3}:z=h(x,t)},
\end{align}
with
\begin{align}\label{eq:Lipschitz}
	\norm{\nabla h(\cdot,t)}_\infty\le1.
\end{align}
Finally, for $t\geq T_g$, the dissipation and height function $h$ satisfy
\begin{align}\label{eq:D_decay}
D(t)\ls \frac{\E\bra{\frac t2}}{t}
\end{align}
and
\begin{align}
  \norm{h}_\infty+t^\frac{1}{3}\norm{\nabla h}_\infty&\lesssim\frac{\V_0+\E_0^{\frac{3}2}}{t^{\frac 23}}, \quad\text{respectively}.\label{eq:h_decay}
\end{align}
\end{theorem}
\begin{remark}[Notation]
In order to get straight to the discussion of the main result, we relegate a discussion of the (hopefully rather natural) notation---including the order notation $\sim,\, \lesssim$ etcetera---to Subsection~\ref{subsec:notation}. The meticulous or skeptical reader may want to start there.
\end{remark}
\begin{remark}(Uniqueness)
	We do not require uniqueness, and the result is true for \emph{any} weak solution. Once the solution has become a graph with small Lipschitz constant, we expect uniqueness to hold (see the discussion in Subsection~\ref{subsec:lit}, below), but we do not require it for our result.
\end{remark}
\begin{remark}[Dimension-dependent rates in the full-space problem]\label{rem:dim_dep}
  Because we work on the full space, there is no spectral gap for the linearized evolution and the algebraic rates are optimal; see Subsection~\ref{subsec:optimality}. Unlike the generic exponential rates that one obtains on compact domains, the algebraic decay rates in~\eqref{eq:E_decay},~\eqref{eq:D_decay}, and~\eqref{eq:h_decay} are dimension-dependent (see also Proposition~\ref{prop:3}) and in this sense reflect \textbf{the structure of the evolution}.
\end{remark}
\begin{remark}[Flat initial data]
The $L^1$ bound~\eqref{eq:V_bounded} reflects the linear behavior when
\begin{align}\label{ao11}
	\E_0\ls \V_0^\frac{2}{3},
\end{align}
in the sense that the nonlinear system obeys the same bound $\V\lesssim \V_0$ satisfied by the linearization (see the discussion in Subsection~\ref{sss:rr} below).
The condition~\eqref{ao11} is a properly nondimensional way of imposing approximate flatness
of the initial configuration. It does so only on average; it does not impose a graph structure,
let alone of uniformly small slope.
\end{remark}
\begin{remark}[The dimensionless quantity]
It is critical for our result that there exists a nondimensional combination of $\E$ and $D$ whose smallness suffices to control the Lipschitz norm of the function $h$. In dimension $d\geq 3$, the slope cannot be controlled in terms of $\E$ and $D$ alone (the quantity $D^3$ \emph{just fails} in $d=3$), and one would need a different approach to obtain Lipschitz control. The exciting observation of our paper in ambient dimension three ($d=2$) is that the elementary but critical fact that the dimensionless quantity $\E\D^2$ becomes small (cf. Lemma~\ref{l:small}) can be leveraged to establish \textbf{generation and propagation of
Lipschitz control} (cf. Section~\ref{S:graph}). In $d=1$, as already exploited in \cite{COW}, smallness of the dimensionless quantity $\E^2\D$ is \textbf{propagated} and serves to maintain Lipschitz control.
\end{remark}

\begin{remark}($d=1$)
Proposition~\ref{prop:3} establishes that also in $d=1$, the solution for appropriate initial data
within the graph setting remains in the graph setting and decays optimally. Additionally Proposition~\ref{prop:1} in $d=1$ shows that the excess mass remains bounded in terms of the initial values as expected (at least for a time quantified in terms of the initial excess energy). Hence ``all'' that is missing for $d=1$ is the analogue of Proposition~\ref{prop:2}.

However we do not believe that the analogue of Theorem~\ref{thm:main} holds for $d=1$. Indeed a heuristic calculation for a flat configuration perturbed only by a round island with radius $R$  at distance $L$  from the interface has dissipation of order $(R^2\log(L/R))^{-1}$. In particular, the dimensionless quantity scales like $\E^2D\sim (\log(L/R))^{-1}$, so that the bound $\E^2D\ll 1$ does not exclude islands if $R\ll L$ and there can be no analogue of Proposition~\ref{prop:2} assuring graph structure based on this smallness condition. In addition one cannot hope for a relaxation rate of the energy as in~\eqref{eq:E_decay}, since the normal velocity at the island can be made arbitrarily small by increasing $L$ (which  changes neither $\E_0$ nor $\V_0$).

\end{remark}
\subsection{Review of literature}\label{subsec:lit}

In the literature, the terms Mullins-Sekerka and Hele-Shaw are sometimes used interchangeably, but we will use the convention from \cite{EO97}. The MS and Hele-Shaw evolutions are related: Their one-phase formulations coincide, while in the two-phase model, the Hele-Shaw model has normal velocity
given by the (continuous over the boundary $\Gamma$) pressure gradient
\begin{align}
	V=\nabla q\cdot n,
\end{align}
where the pressure is harmonic outside $\Gamma$ and has jump given by the mean curvature on the surface
\begin{align}
	-\Delta q&=0,\quad \text{in }\R^{d+1}\setminus \Gamma\\
	[q]&=H,\quad \text{on }\Gamma.
\end{align}
We refer to \cite{EO97} for details.

We concentrate here on literature about the two-phase Mullins-Sekerka evolution and do not mention results in neighboring areas like the one-phase Mullins-Sekerka evolution, the mean curvature flow, and the Stefan problem with Gibbs-Thomson boundary condition. We note but do not comment additionally on the rich connection to the Cahn-Hilliard equation as a diffuse interface approximation of MS and corresponding convergence results. For an excellent, broad overview, we refer to the introduction of \cite{HS22}.

\paragraph{Existence} Most of the available existence results for MS are on bounded domains. We comment separately on short-time and long-time results.

Short-time existence results for MS on bounded domains go back to \cite{Che93} for weak solutions in $d+1=2$ dimensions, and to \cite{CHY96,ES97} for strong solutions in arbitrary dimension. As mentioned above, starting from  non-connected initial data, components can merge or vanish. Even worse: starting from connected initial data, the evolution can produce non-connected phases. Thus the evolution is expected to be \emph{non-smooth} at least at finitely many points in time, which poses a natural restriction on the existence time of strong solutions. Very recently, \cite{EMM23} proves short-time existence of strong solutions in the whole space in the graph setting in $d+1=2$. This is the only existence result on unbounded domains of which we are aware.

Even though non-smooth events are expected during the evolution and need to be incorporated in a weak solution concept, the important properties of surface reduction and (signed) mass conservation are maintained across these events. Long-time existence of weak solutions via approximation by minimizing movements is proved in the BV setting in \cite{LS95} with an additional assumption on the convergence of the energy. Building on results by Sch\"atzle \cite{Sch}, the existence of varifold solutions that ``contain'' BV solutions is proved in \cite{Roe05} without the additional assumption. The recent work \cite{HS22} provides long-time existence of weak gradient flow solutions that satisfy a De Giorgi type inequality for energy and dissipation; in particular an inequality corresponding to \eqref{eq:D_int} (with factor $\tfrac 12$) holds true.

A combination of \cite{EMM23}, \cite{HS22}, and a work in preparation \cite{FHLS} on weak-strong uniqueness  is our justification for Hypothesis \textbf{(H)}: Due to \cite{HS22} it seems reasonable that weak solutions satisfying the correct dissipation behavior also exist on the whole space. An extension of \cite{EMM23} to $d+1=3$ dimensions seems possible albeit challenging; in this case the long-time existence of a strong solution would be guaranteed by our uniform control of the dissipation, and weak-strong uniqueness would guarantee that a solution remains smooth once it has become regular.

\paragraph{Ostwald ripening}

Ostwald ripening consists of the coarsening of initial configurations in which  many small, nearly spherical islands of one of the phases are present. The larger islands grow at the expense of the smaller ones, which eventually disappear. During this process the number of connected components decreases and the average size of islands grows until only the largest component (in our case unbounded) remains. We refer to \cite{Nie99,AF03} and the references mentioned therein for a detailed analysis of this stage of the evolution.

Our result includes the regime of Ostwald ripening: One can consider initial configurations in which the unbounded components of positive and/or negative phase are complemented by spherical or nearly spherical islands of the opposite phase---or initial data that flows into such a regime.

\paragraph{Relaxation rates} We are interested in capturing and quantifying convergence to the longtime limit. On bounded domains, convergence to the groundstate is exponential (see Remark~\ref{rem:dim_dep}) since the linearization exhibits a spectral gap. This is made rigorous in \cite{Che93, ES98} where exponential convergence is proved (in $d+1=2$ and $d+1\geq 2$, respectively) for initial data that are close to a ball in a strong sense. In \cite{JMPS} exponential convergence to a collection of equally sized discs is proved for a so-called flat flow solution of MS on the torus of dimension $d+1=2$. They consider a large class of initial data and their method relies on a quantitative Alexandrov Soap Bubble Theorem.
An existence and stability analysis for an interface with $\frac{\pi}2$ contact angle at the boundary of the domain is conducted in \cite{ARW,GR} (in $d+1=2,\,3$ and $d+1=2$, respectively), including exponential stability for minimizers in appropriate situations.

Regarding convergence rates for the whole space problem, where no spectral gap is available, a relaxation result in $d+1=2$ is established in \cite{COW} in terms of the energy gap, the dissipation, and the squared $\Hd^{-1}$ distance
\begin{align}
\Hdist\coloneqq \norm{\chi(0)}^2_{\Hd^{-1}}
\end{align}
associated to the gradient flow. The assumptions on the initial data are  $\E_0^2 D_0\ll 1$, that $\Gamma_0$ is the graph of a function $h(0)$ with $\norm{h_x(0)}_\infty \le 1$, and
\begin{align}
\Hdist_0\coloneqq \norm{\chi(0)}^2_{\Hd^{-1}}<\infty
\end{align}
(rather than $\V_0<\infty$).
The method exploits the interpolation inequality
\begin{align*}
\E\lesssim \Hdist^{\frac 12}D^{\frac 12}
\end{align*}
to obtain the decay rates
\begin{align}
\E\ls \frac{\Hdist_0}{t}, \quad D\ls \frac{\Hdist_0}{t^2},\label{cow}
\end{align}
which are the same algebraic decay rates as pointed out by Brezis \cite{Bre} for a gradient flow with respect to a \emph{convex energy}.
A corollary is the control
\begin{align*}
\norm{h}_\infty+t^{\frac 13}\norm{\nabla h}_{\infty}\ls\frac{\Hdist_0}{t^{\frac 13}}.
\end{align*}

Although  the  $t^{-1}$ decay rate of the energy from \eqref{cow} is the same algebraic rate that we will obtain for $d=1$ in
Proposition~\ref{prop:3} below,
the $L^1$ approach yields a stronger result for compactly supported perturbations, since it allows non-neutral perturbations of the half\-space, whereas $\Hdist_0<\infty$ implies $\int_{\R^d} h_0\dx=0$.
In $d=2$, our $L^1$ result picks up the faster dimension-dependent decay, while the method of \cite{COW} is blind to the dimension.

The most important distinction to be drawn between \cite{COW} and the present work is that here, for $d=2$, we show that the graph structure and Lipschitz bound on the height are \emph{generated by the dynamics}, whereas in \cite{COW} this is an assumption on the initial data. We remark that convergence rate results for Mullins-Sekerka for non-perturbative, non-graph initial data are extremely rare; to the best of our knowledge, there is only \cite{JMPS} for flat flow solutions in two dimensions (without resolving dependence of the rate on the initial data) and the present work.

The $L^1$ method used here was  previously introduced to establish relaxation results for the one-dimensional Cahn-Hilliard equation \cite{OSW,BSW}, and is shown here to be a useful tool for obtaining sharp relaxation rates for geometric evolutions.

\subsection{Discussion of optimality}\label{subsec:optimality}

The rates in \eqref{eq:h_decay} are optimal: It is exactly these rates that hold for the geometrically linearized problem. In order to see this, consider
\begin{alignat}{2}
    -\belta f&=0  \quad&\text{in }\R^d\times \set{z>0},\label{eq:lin1}\\
    f&=\Delta h \quad&\text{on }\R^d\times \set{z=0},\label{eq:lin2}\\
    h_t&=-2f_z \quad&\text{on }\R^d\times \set{z=0},\label{eq:lin3}\\
    h(0)&=h_0\quad &\text{on }\R^d\times \set{z=0},\label{eq:lin4}
\end{alignat}
where the flat geometry is fixed and decoupled from $h$. The boundary values for $f$ are imposed by $h$ in form of the linearized expression for the mean curvature (cf. \eqref{eq:H}) and the jump of the normal derivative of $f$ at the flat interface determines the linearization of the normal velocity, which is just $-h_t$ (cf. \eqref{eq:norm_vel}).

In order to solve problem \eqref{eq:lin1}--\eqref{eq:lin4}, we Fourier transform \eqref{eq:lin1} and \eqref{eq:lin2} in the $x$ variable and deduce: \begin{align*}
-\partial_{zz}^2\hat f(k,z)&=-\abs{k}^2\hat f(k,z),\\
\hat f(k,0)&=-\abs{k}^2\hat h(k),
\end{align*}
where $k$ is the tangential wavenumber and $\hat{f}$ and $\hat{h}$ are the Fourier-transform of $f$ and $h$ in the tangential direction, respectively. Combining this with the growth condition on $\hat{f}(k,z)$ from $\norm{\nabla f}_{L^2(\R^{d+1}\setminus\{z=0\})}<\infty$, we obtain
\begin{align*}
\hat{f}(k,z)=-\exp(-\abs{k}z)\abs{k}^2\hat{h}(k),
\end{align*}
Finally, using \eqref{eq:lin3}, we arrive at the closed equation
\begin{align}\label{eq:Fourier}
\partial_t \hat{h}+2|k|^3\hat{h}=0.
\end{align}
By the above computation we see that the Dirichlet-to-Neumann map $A:f\mapsto -f_z$ on $\R^d$ is represented by the multiplier $\abs{k}$ in Fourier-space which is why we denote $\abs{\nabla}\coloneqq A$. With this notation \eqref{eq:Fourier} can be written in physical space as
\begin{align}\label{eq:linearization}
\partial_t h-2|\nabla|\Delta h=0,
\end{align}
and the Fourier symbol of $-2|\nabla|\Delta$ is exactly given by $2|k|^3$ in the tangential wave number $k$. From the Fourier representation
 one obtains
\begin{align*}
\hat{h}(t)=\hat{G}(t)\hat{h}_0,
\end{align*}
with $\hat{G}(t,k)=\exp(-2\abs{k}^3t)$. Thus, $h$ itself is given by
\begin{align*}
h(t)=G(t)\ast h_0,
\end{align*}
where
\begin{align}\label{eq:profile}
G(t,x)=\frac{1}{t^{\frac d3}}\tilde G\bra{\frac{x}{t^{\frac 13}}}
\end{align}
for some profile $\tilde G=G(1)$. In particular
\begin{align*}
 \norm{\hat{G}(t)}_1\ls t^{-\frac d3},\qquad \norm{\;\abs{k}\hat{G}(t)}_1\ls t^{-\frac{d+1}3},
\end{align*}
 which together with the estimates
\begin{align*}
 \norm{h}_\infty&\le \norm{G}_\infty\norm{h_0}_1\ls \norm{\hat G}_1\norm{h_0}_1,\\
 \norm{\nabla h}_\infty&\le \norm{\nabla G}_\infty\norm{h_0}_1\ls \norm{\; \abs{k}\hat G}_1\norm{h_0}_1,
\end{align*}
yields
\begin{align}\label{eq:h_decay_int}
 \norm{h}_\infty+ t^{\frac 13} \norm{\nabla h}_\infty\lesssim t^{-\frac{d}{3}}\int_{\R^d}|h_0| \dx.
\end{align}
Hence initial data in $L^1$ lead to an $L^\infty$ decay of $t^{-\frac{d}{3}}$, which is precisely the decay captured in \eqref{eq:h_decay}.

\subsection{Method}\label{subsec:method}
There are two things to explain: Where does the graph structure come from and how are the relaxation rates obtained?

\subsubsection{Graph structure}
We begin with the former. It is an elementary observation that
\begin{lemma}\label{l:small}
  For every $\eps$ there exists $T_g\leq \frac{2}{3\sqrt{\eps}}\E_0^{\frac 32}$ such that
\begin{align}\label{eq:small_quantity}
\bra{\E D^2}(T_g)\leq \eps.
\end{align}
\end{lemma}
\begin{proof}
The proof follows using $T= \frac{2}{3\sqrt{\eps}}\E_0^{\frac 32}$ and the inequality almost everywhere in time:
\begin{align}
(\E D^2)^{\frac 12}=\E^{\frac{1}{2}}D\overset{\eqref{eq:Edot}}\leq-\frac 23\ddt \E^{\frac 32}.\label{id}
\end{align}
Then since $\E$ and hence $\E^\frac{3}{2}$ is nonincreasing, we deduce
\begin{align}
\inf_{t\leq T}\big(\E D^2\big)^\frac 12\leq \frac{1}{T}\int_0^T (\E D^2)^{1/2}\dt\overset{\eqref{id}} \leq \frac 2{3T} \E_0^{\frac 32}= \eps^\frac 12.
\end{align}
\end{proof}
For $\eps$ sufficiently small,  $\E D^2\le \eps$ implies graph structure
of $\Gamma$ and the Lipschitz control
\begin{align}\label{eq:Lip_control_int}
	\norm{\nabla h}_\infty\ls (\E D^2)^{\frac 16},
\end{align}
which we show using on the one hand results of Meyers \& Ziemer \cite{MZ,Zie} and Schätzle~\cite{Sch} to establish local control on the $L^p$ norm of the mean curvature $H$ and on the other hand Allard's regularity theory \cite{All} to convert $L^p$ control of the curvature into graph structure and Lipschitz continuity with respect to the plane $\{z=0\}$. An interpolation between small scale control of $H$ via $D$ and large scale control of the comparison plane via $\E$ yields exactly \eqref{eq:Lip_control_int}.

Based on \eqref{eq:Lip_control_int},  we will distinguish between an initial layer in time, where boundedness of the energy will suffice, and the small slope regime
\begin{align}\label{ao08}
\norm{\nabla h}_\infty\ll 1,
\end{align}
where decay of  energy and dissipation will be established and exploited.
In the regime \eqref{ao08}, we
leverage the control of $D$, $\E D^2$ to deduce
\begin{align}
	\norm{\nabla^2h}_p\lesssim \E^{\frac{4-p}{6p}}D^{\frac{2(p-1)}{3p}}\label{ao12}
\end{align}
for a suitable $p\in (2,4)$. This we derive by appealing to a trace-Sobolev estimate and a Meyers-type perturbative argument.

We then consider the dissipation of the dissipation; based on \eqref{ao12}, we
establish the differential inequality
\begin{align}
\ddt D\lesssim D^4.\label{eq:ddt_D_int}
\end{align}
Together with \eqref{eq:Edot}, this immediately implies
\begin{align}\label{eq:ddt_E2D}
  \ddt(\E D^2)\leq 0\qquad\text{for }\; t\geq T_g,
\end{align}
so that
$\E D^2\ll 1$ and hence also graph structure and the small slope regime are preserved.

\subsubsection{Relaxation rates via Nash and \texorpdfstring{$L^1$}{} control}\label{sss:rr}
As far as the relaxation in time, the core idea is to use a Nash-type estimate controlling $\E$ in terms of the $L^1$-distance $\V$ and the dissipation $D$:
\begin{align}\label{eq:nash_int}
\E\lesssim \V^{\frac{6}{7}}D^{\frac{4}{7}},
\end{align}
which is not hard to show.
If $\V$ remains bounded, then this algebraic relation combined with the differential equation~\eqref{eq:Edot} yields a differential inequality for the excess energy, which implies \eqref{eq:E_decay}. Hence our main mathematical contribution is a duality argument that shows that $\V$ does
not increase too much:
\begin{align}\label{eq:V_bound_int}
\V\lesssim \V_0+\E_0^{\frac 32}.
\end{align}

One can view this result in the following way: Although the evolution is \emph{not initially in a perturbative regime} and thus behaves genuinely nonlinearly, the gradient flow structure imbues the problem with an inherent relaxation towards the linear regime. We show with the duality argument that one can leverage this relaxation
to capitalize on the linearization \emph{even in the initial nonlinear stage}---essentially because the evolution \emph{does not move away from the linear regime that much}.

\medskip

In both the initial layer and the small slope regime, we employ a duality argument inspired by \cite{NV} and used previously in the $L^1$-method developed and employed in \cite{OSW,BSW}. The earliest use of this duality method may be the adjoint method of Evans used by Evans, Tran, and others in the context of Hamilton-Jacobi equations; see \cite{E} and the citing references.

The duality argument is a nonlinear generalization of a duality argument for the
linearization \eqref{eq:linearization}, which we will explain in the remainder of this subsection. The main
task for the duality arguments of our paper will hence be to estimate the linearization errors in the right way.

But for now, let us look at the linear case. To see that \eqref{eq:V_bound_int} is true for the linear problem in any dimension, i.e., that the $L^1$-norm of $h$ is controlled in terms of the initial data, it is enough to use
\begin{align*}
 \norm{h}_1&\le \norm{G}_1\norm{h_0}_1
\end{align*}
in combination with the fact that $G$ is bounded in $L^1$ uniformly in $t$. This in turn follows from \eqref{eq:profile} and the $L^1$ control of the mask $\tilde G$,
which is smooth and
decays with rate
\begin{align}
\abs{x}^{-d-1} \quad\text{for large $x$.}\label{Gcay}
\end{align}
The decay can be derived via integration by parts as in
\begin{align}
 (2\pi)^{\frac d2}\abs{x}^{d+1}\tilde G(x)&=\int_{\R^d}\abs{x}^{d+1}e^{ix\cdot k}e^{-2\abs{k}^3}\dk\label{eq:G1}\\
 &=-6i\int_{\R^d}\abs{x}^{d-1}e^{ix\cdot k}x\cdot k\abs{k}e^{-2\abs{k}^3}\dk,
\end{align}
repeating $d$ more times and observing that $\abs{k}^3$ is $(d+1)$-times weakly differentiable.

This approach is not well-suited for generalization to nonlinear equations. Hence we turn instead to the dual characterization of the $L^1$ norm via
\begin{align*}
\norm{h(T)}_1=\int_{\R^d} \abs{h(T)}\dx=\sup_{\psi\in L^\infty(\R^d),\norm{\psi}_\infty\le 1}\int_{\R^d}h(T)\psi\dx
\end{align*}
and the dual or adjoint equation with terminal data $\psi$:
\begin{align}
\partial_t {\uu}+2\abs{\nabla}\Delta {\uu}&=0 \quad \text{on }[0,T)\times \R^d,\label{eq:dual_lin1}\\
{\uu}(T)&=\psi \quad \text{on } \R^d,\notag
\end{align}
for which solutions exist, which can be seen by applying the kernel $G$ backwards in time. For reference below, the property $\norm{G}_1<\infty$ yields
\begin{align}
\norm{{\uu(t=0)}}_\infty\ls \norm{\psi}_\infty.\label{eq:zeta_psi}
\end{align}

The idea is now to introduce the harmonic extension $\ub$  of  ${\uu}$ to $\R^d\times \set{z>0}$, which satisfies
\begin{alignat}{2}
-\partial_t{\ub}+2\partial_z\Delta {\ub}&=0& \quad &\text{on } (0,T)\times \R^d,\label{eq:uv1}\\
-\belta {\ub}&=0 & \quad &\text{on } (0,T)\times \R^d\times \set{z>0},\label{eq:uv2}\\
{\ub}(T)&=\psi& \quad &\text{on } \R^d,\label{eq:dual_lin2}
\end{alignat}
and analogously for the harmonic extension $\hb$ of the function $\hh$ satisfying \eqref{eq:linearization}.

Using \eqref{eq:uv1} and the corresponding equation for $\hb$,  we calculate
\begin{align}
\ddt\int_{\R^d} {\uu} h\dx&=\ddt\int_{\R^d} \ub \hb\dx = 2\int_{\R^d}\bra{\partial_z\Delta \ub}\hb-\ub(\partial_z\Delta \hb)\dx\nonumber\\
&= 2\int_{\R^d}\partial_z \ub\Delta \hb-\ub\partial_z\Delta \hb\dx\nonumber\\
\overset{\eqref{eq:uv2}}&{=} -2\int_{\bra{z>0}}\babla \ub\cdot\babla\Delta \hb-\babla \ub\cdot \babla\Delta \hb\dx\nonumber\\
&=0,\label{WHY}
\end{align}
and deduce
\begin{align*}
\int_{\R^d}h(T)\psi\dx=\int_{\R^d}h_0{\uu}(0)\dx
\overset{\eqref{eq:zeta_psi}}{\ls} \norm{h_0}_1\norm{\psi}_\infty.
\end{align*}
Taking the supremum over all $\psi$ with $\norm{\psi}_\infty\le 1$ gives $\norm{h}_1\ls \norm{h_0}_1$. Applying this argument to a nonlinear perturbation of \eqref{eq:linearization} leads to a nontrivial right-hand side in \eqref{WHY}, for which suitable estimates are required. The advantage of the method is that it isolates and quantifies the nonlinear behavior, since the linear contributions cancel. This is the approach that we will use to establish $L^1$ control.

\subsection{Notation and organization}\label{subsec:notation}

\begin{notation}[Time-dependent functionals]
We explicitly denote the $t$-dependence of the quantities $\Gamma,\, \E,\, D,\, \V$ only when we occasionally want to draw attention to it.
We denote the initial values with index $0$:
  \[\Gamma_0,\quad \E_0,\quad D_0,\quad \V_0. \]
\end{notation}

\begin{notation}[Gradients and such]\label{not:normal_bold}   We use regular and boldface font to distinguish between $d$-dimensional and $(d+1)$-dimensional objects. We use $\dx$ to note integration with respect to the $d$-dimensional Lebesgue measure over $x\in \R^d$ and $\nabla$, $\Delta$ to denote the gradient and Laplacian with respect to $x$; for the $(d+1)$-dimensional measure, gradient, and Laplacian, we use $\bdx$, $\babla$, and $\belta$:
\begin{align*}
\bdx=\dx\,dz,\qquad \;\babla f=\begin{pmatrix}\nabla f\\ \partial_z f\end{pmatrix},\qquad \belta:=\Delta+\partial_{zz}.
\end{align*}
We refer to $x$ as the tangential variable.
\end{notation}
\begin{notation}[Norms] Our convention is to denote norms on the interface explicitly and suppress the domain on $\R^d$, e.g.,
\begin{align*}
 \norm{g}_{L^p(\Gamma)} , \qquad \norm{g}_{p}=\norm{g}_{L^p(\R^d)},
\end{align*}
except when $\R^d$ is made explicit for emphasis.

\end{notation}

\begin{notation}[Jump]\label{not:jump}
For a quantity $q:\R^d\to \R$ that is defined on both $\Omega_+$ and $\Omega_-$ and has a well-defined trace $q_+$ and $q_-$ on $\Gamma$ from each side, respectively, we denote the jump by
\begin{align*}
	[q]:=q_+-q_-.
\end{align*}
\end{notation}
\begin{notation}[Minimum]\label{not:min}
We occasionally use the notation
\[A\wedge B:=\min\{A,B\}.\]
\end{notation}
\begin{notation}[Order]\label{not:1}
We use the notation

	\begin{align*} A \lesssim B \end{align*}
if $B\geq 0$ and there exists a universal constant $C\in(0,\infty)$ depending at most on the dimension $d$, such that
$A \leq CB$. If $A \lesssim B$ and $B \lesssim A$ we write $A \sim B$.

We say
\begin{align*} A \lesssim B \quad\text{implies}\quad E\lesssim F\end{align*}
if for every $C_1<\infty$ there exists $C_2<\infty$ such that
\begin{align*} A \leq C_1\, B \quad\text{implies}\quad E\leq C_2\, F,\end{align*}
and analogously for statements involving $A\sim B$ or $E\sim F$.
\end{notation}
\begin{notation}
  When we say let $\V$ and $\E$ be finite, we mean that there exists a set $\Omega_+$ of locally finite perimeter with reduced boundary $\Gamma$ and that the associated excess mass and excess energy are finite.
\end{notation}
\paragraph{Organization.}
The rest of the paper is organized as follows. Section~\ref{S:main} announces the three central propositions (which will be proved in Sections~\ref{sec:initial}, \ref{S:graph}, and \ref{S:relax2}) and establishes the main theorem based on these results. Section~\ref{sec:initial} establishes control of the excess mass in the initial layer. Section~\ref{S:graph} proves that the graph regime and Lipschitz control are reached by time $T_g$. Finally, Section~\ref{S:relax2} establishes control and decay for later times $t>T_g$.
A few auxiliary results are gathered in the appendix.

\section{Central propositions and proof of main theorem}\label{S:main}

As described in Subsection~\ref{subsec:method}, in the initial layer we will use merely $\E\leq \E_0$ but need an argument to control $\V$; the result is recorded in Proposition~\ref{prop:1}.
\begin{proposition}\label{prop:1}
Let $d=2$ or $d=1$. Within the setting of Theorem~\ref{thm:main} and for any $T\lesssim \E_0^{\frac 3d}$, there holds
  \begin{align}
    \sup_{0\leq t\leq T}\V(t)\lesssim \V_0+\E_0^{\frac{d+1}{d}}.\label{eq:V_control_init}
  \end{align}
\end{proposition}
The next main ingredient in $d=2$ is to show that a graph structure is achieved once $\E\D^2$ is small enough.
\begin{proposition}\label{prop:2}
Let $d=2$. Let $\Gamma$ be the reduced boundary of a set $\Omega_+$ of locally finite perimeter and suppose finite excess energy $\E$ and dissipation $D$. There exists $\eps_1>0$ such that if $\E \D^2\leq \eps_1$, then there exists a Lipschitz function $h:\R^2\to\R$ such that
\begin{align}\label{propsoup}
\Gamma=\{(x,z)\colon x\in \R^2,\; z=h(x)\}\qquad\text{and}\qquad	\norm{\nabla h}_\infty\ls (\E \D^2)^{\frac 16}.
\end{align}
\end{proposition}

We will use Lemma~\ref{l:small} with an $\eps$ small enough so that Proposition~\ref{prop:2} identifies a time $T_g\lesssim\E_0^{3/2}$ such that graph structure is achieved.
In addition, as described in Subsection~\ref{subsec:method}, we will choose $\eps$ sufficiently small so that \eqref{eq:ddt_E2D} and the control from \eqref{propsoup} locks the dynamics within the small slope regime. The next and final proposition encapsulates relaxation to flat for graphs with small Lipschitz norm. We use this in Theorem~\ref{thm:main} in $d=2$ to establish relaxation once the evolution has entered the graph setting; in $d=1$, the proposition says that for initial data as given, the relaxation rates hold. We state it as a ``stand-alone result'' since it already gives a stronger relaxation result for the graph setting in $d=1$ than has previously been established (see Subsection~\ref{subsec:lit}).
\begin{proposition}\label{prop:3}
Let $d=2$ or $d=1$.
There exists $\eps_2>0$ with the following property. Consider the the MS dynamics under the hypothesis \textbf{(H)} in the graph setting, i.e., where $\Gamma_0$ is the graph of a function $h_0:\R^d\to\R$. If $\E_0^{3-d} D_0^d \leq \eps_2$ and $\norm{\nabla h_0}_{\infty}\leq 1$, then $\Gamma$ remains a graph satisfying  $\norm{\nabla h}_{\infty}\leq 1$   and
\begin{align}
\V&\ls \V_0+\E_0^{\frac {d+1}d},\label{st1}\\
\E&\ls \min\set{\E_0,\frac{\V_0^2+\E_0^\frac{2(d+1)}{d}}{t^{\frac {d+2}3}}}\label{st2}
\end{align}
hold for all future times. In addition there exists $T\sim\E_0^{\frac {d+1}{d}}$ such that
\begin{align}
D(t)&\ls \frac{\E\bra{\frac t2}}{t},\label{st3}\\
\norm{h}_\infty+t^\frac{1}{3}\norm{\nabla h}_\infty&\lesssim\frac{\V_0+\E_0^{\frac{d+1}d}}{t^{\frac d3}}      \label{st4}
\end{align}
holds for all $t\ge T$.
\end{proposition}
\begin{proof}[Proof of Theorem~\ref{thm:main}]
We fix $\eps:=\min\{\eps_1,\eps_2\}$ for the constants $\eps_1$ and $\eps_2$ from Propositions~\ref{prop:2} and \ref{prop:3}. If necessary we reduce $\eps$ additionally so that for the implicit constant on the right-hand side of \eqref{propsoup}, there holds
\begin{align*}
  \norm{\nabla h}_\infty\leq C (\E D^2)^{\frac 16}\leq 1.
\end{align*}
For this $\eps>0$, we apply Lemma~\ref{l:small} to define $T_g\lesssim \E_0^{3/2}$ so that
\begin{align*}
  \E D^2(T_g)\leq \eps.
\end{align*}
On $[0,T_g]$, we use Proposition~\ref{prop:1} to control
\begin{align}
    \sup_{0\leq t\leq T_g}\V(t)\lesssim \V_0+\E_0^{\frac 32}.\label{VgTg}
  \end{align}
By choice of $\eps$,
Proposition~\ref{prop:2} yields that the interface at time $T_g$ is a graph, and for $t\geq T_g$ Proposition~\ref{prop:3} yields that the excess mass is bounded by
\begin{align}\notag
\V\ls \V(T_g)+\E(T_g)^{\frac 32}\overset{\eqref{eq:E_bound_simple}, \eqref{VgTg}   }\lesssim\V_0+\E_0^{\frac 32}\qquad \text{for $t\geq T_g$},
\end{align}
and the energy decays according to
\begin{align}\label{latetime}
\E&\ls \min\set{\E_0,\frac{\V(T_g)^2+\E(T_g)^{3}}{(t-T_g)^{\frac {4}3}}}\overset{\eqref{eq:E_bound_simple}, \eqref{VgTg}   }\lesssim\min\set{\E_0,\frac{\V_0^2+\E_0^{3}}{(t-T_g)^{\frac {4}3}}}
\quad \text{for $t\geq T_g$}.
\end{align}
In view of
\[\frac{\E_0^{3}}{t^{\frac{4}{3}}}\gtrsim \E_0\qquad\text{for }t\lesssim \E_0^{\frac 32},\]
\eqref{eq:E_bound_simple} and~\eqref{latetime} combine to give~\eqref{eq:E_decay}.

Possibly increasing $T_g$ by $T$ from Proposition~\ref{prop:3} and again arguing as for~\eqref{VgTg} to control $\V$ up to this point, we use \eqref{st3}-\eqref{st4} to  obtain~\eqref{eq:D_decay}-\eqref{eq:h_decay}.
\end{proof}

\section{Control of the initial layer}\label{sec:initial}

This section is devoted to the proof of Proposition~\ref{prop:1}. Throughout the section we will let $d\in\{1,2\}$ and
$\chi$ be as defined in \eqref{eq:vfinite}.

We represent $\V$ via duality in the form
\begin{align}\label{eq:V_dual_smooth}
	\V=\sup_{\psi\in L^\infty(\R^{d+1}), \norm{\psi}_{\infty}\le 1}\int_{\R^{d+1}}\psi \chi \bdx=\sup_{\psi\in C^\infty_c(\R^{d+1}), \norm{\psi}_{\infty}\le 1}\int_{\R^{d+1}}\psi \chi \bdx,
\end{align}
because it will be convenient below to work with smooth and compactly supported test functions. (The restriction can be justified by passing to the limit in a standard cut-off and mollification argument.)
Rather than work with test functions on $\R^{d+1}$, it will be convenient to argue for \eqref{eq:V_control_init} by way of the following modified version of the excess mass
\begin{align}\label{eq:Vb_def}
\Vb\coloneqq \sup_{\psi\in L^\infty(\R^d), \norm{\psi}_{\infty}\le 1}\int_{\R^{d+1}}\bar \psi \chi \bdx =\sup_{\psi\in C^\infty_c(\R^d), \norm{\psi}_{\infty}\le 1}\int_{\R^{d+1}}\bar \psi \chi \bdx,
\end{align}
where $\bar \psi$ is the harmonic extension of $\psi$ (and the second equality is justified as in \eqref{eq:V_dual_smooth}).
Control of $\Vb$ will deliver control of $\V$ using
\begin{lemma}\label{l:V_Vb}
If $\V,\E<\infty$, then
	\begin{align}\label{eq:V_Vb}
		\V\ls \Vb+\E^{\frac{d+1}{d}}.
	\end{align}
\end{lemma}

As explained in Section~\ref{subsec:method}, the idea to bound $\Vb$ is to introduce the adjoint harmonic extension $\ub$ of $\psi$ using \eqref{eq:uv1}-\eqref{eq:dual_lin2} and
to estimate  $\ddt \int_{\R^{d+1}}\chi\bar u\bdx$. In the following auxiliary lemma we split the error and announce the corresponding estimates. We denote by $\tilde u$ and $\tilde v$ the \emph{constant extensions in the $z$-direction} of $u$ and
$v=-\abs{\nabla}u$
to $\R^{d+1}$.
\begin{lemma}[Splitting the error and preprocessing]\label{l:preproc_init}
Let $T>0$, $\psi\in C^\infty_c(\R^d)$ with $\norm{\psi}_\infty\le 1$, and let $\bar{u}$ satisfy \eqref{eq:uv1}--\eqref{eq:dual_lin2}, extended by even reflection to $\R^d\times \{z<0\}$. There holds
\begin{align}
\lefteqn{\bra{\int_{\R^{d+1}}\chi \bar u \bdx}(T)-\bra{\int_{\R^{d+1}}\chi \bar u \bdx}(0)}\\
&\!\!\!\!=\int_0^T \left(2\int_{\R^{d+1}} \chi\tilde v \partial_z f\;\bdx-2\int_{\Gamma}(1-e_z\cdot n)(\babla\tilde v\cdot n)\dS+\int_{\R^{d+1}}\chi(\partial_t\bar u-\partial_t \tilde u)\bdx\right)\dt.\label{eq:ddt}
\end{align}
Moreover, the error terms can be estimated for almost every $t$ as
\begin{align}
  A_1&\coloneqq \abs{2\int_{\R^{d+1}} \chi\;\tilde v\, \partial_z f\,\bdx}\ls \frac{\V^{\frac 12}D^{\frac 12}}{(T-t)^{\frac 13}},\label{eq:A1_est}\\
 A_2&\coloneqq \abs{2\int_{\Gamma}(1-e_z\cdot n)(\babla\tilde v\cdot n)\dS}\ls  \frac{\E}{(T-t)^{\frac 23}}, \label{eq:A2_est}\\
A_3&\coloneqq \abs{\int_{\R^{d+1}}\chi\;(\partial_t\bar u-\partial_t \tilde u)\,\bdx}\ls   \frac{\V^{\frac 12}\E^{\frac 12}}{(T-t)^{\frac 56}}+\frac{\E}{(T-t)^{\frac 23}}.\label{eq:A3_est}
\end{align}
\end{lemma}
For the proof of Proposition~\ref{prop:1} it will be convenient to use the notation:
\begin{align}\label{eq:V_T}
\V_T\coloneqq\sup_{t\in [0,T]}\V(t),\qquad\qquad \Vb_T\coloneqq\sup_{t\in [0,T]}\Vb(t).
\end{align}

\begin{proof}[Proof of Proposition~\ref{prop:1}]
For given $\psi\in C_c^\infty(\R^d)$ with $\norm{\psi}_\infty\le 1$ and $T>0$, let $\bar u$ be the solution to \eqref{eq:uv1}--\eqref{eq:dual_lin2}. Using Lemma~\ref{l:preproc_init} we write
	\begin{align}
		\bra{\int_{\R^{d+1}} \bar \psi \chi\bdx}(T)&\le\bra{\int_{\R^{d+1}}\bar u\chi \bdx}(0)+\int_0^T A_1+A_2+A_3\dt.	
	\end{align}	
	Employing the estimates \eqref{eq:A1_est}--\eqref{eq:A3_est}, taking the supremum over $\psi$, and using \eqref{eq:u_decay1} in the form $\norm{\bar u}_\infty\le \norm{\psi}_\infty\le 1$, we obtain
	\begin{align}\label{eq:Vb_est}
		\Vb(T)\le \V_0+C\int_0^{T} (T-t)^{-\frac 13}\V^{\frac 12}D^{\frac 12}+(T-t)^{-\frac 23}\E+(T-t)^{-\frac 56}\V^{\frac 12}\E^{\frac 12}\dt.
	\end{align}
It remains to estimate the time integral, which we do term by term:
	\begin{align}
		\int_0^{T}(T-t)^{-\frac 13}\V^{\frac 12}D^{\frac 12}\dt &\ls \V_{T}^{\frac 12}\bra{\int_0^{T} (T-t)^{-\frac 23}\dt \int_0^{T}D\dt}^{\frac 12}\ls \V_{T}^{\frac 12}\bra{T^{\frac 13}\E_0}^{\frac 12},\\
		\int_0^{T} (T-t)^{-\frac 23}\E\dt&\ls  \E_0\int_0^{T}(T-t)^{-\frac 23}\dt\ls T^{\frac 13}\E_0,\\
		\int_0^{T} (T-t)^{-\frac 56}\V^{\frac 12}\E^{\frac 12}\dt &\ls \V_T^{\frac 12}\E_0^{\frac 12}\int_0^{T}(T-t)^{-\frac 56}\dt \ls \V_T^{\frac 12}\E_0^{\frac 12}T^{\frac 16}.
	\end{align}
	Inserting the right-hand sides with $T\lesssim \E_0^{\frac 3d}$ into \eqref{eq:Vb_est}, using Lemma~\ref{l:V_Vb}, and applying Young's inequality yields
	\begin{align}
		\Vb(T)\le \V_0+\frac 12\Vb_{T}+C\E_0^{\frac {d+1}{d}}.
	\end{align}
	Taking the supremum in $T$, absorbing the second term from the right-hand side in the left-hand side, and again applying Lemma~\ref{l:V_Vb} yields the result.
\end{proof}

\subsection{Proofs of auxiliary statements}
For the proof of Lemma~\ref{l:V_Vb}, we will need one technical lemma, which we state here and prove after establishing Lemma~\ref{l:V_Vb}.
\begin{lemma}\label{l:EV_interpolation}
If $\V, \E<\infty$, then
\begin{align}\label{eq:n2_est}
	\int_{\Gamma}|n^\prime|^2\dS\ls \E,
\end{align}
where $n^\prime=n-(n\cdot e_z)e_z$.
Furthermore, for $R>0$, there holds
\begin{align}\label{eq:zR}
	\int_{\Gamma}(|z|\wedge R)^2\dS\ls R\V+R^2\E,
\end{align}	
where we recall the notation $\abs{z}\wedge R=\min\{\abs{z},R\}$.
\end{lemma}
\begin{proof}[Proof of Lemma~\ref{l:V_Vb}]
It will be useful to introduce
the  intermediary functional
\begin{align}
	\Vt\coloneqq \sup_{\psi\in L^\infty(\R^d),\norm{\psi}_\infty\le 1}\int_{\R^{d+1}} \tilde \psi \chi \bdx,
\end{align}
where as usual $\tilde \psi$ denotes the constant extension of $\psi$ and we again recall the definition of $\chi$ from \eqref{eq:vfinite}.

The proof consists of several steps and makes use of a cut-off lengthscale $R>0$, corresponding to which we define
\begin{align*}
  \V_R:=\int_{\R^d\times \{\abs{z}\le R\}} |\chi|\,\bdx,\qquad \Vt_R=\sup_{\psi\in L^\infty(\R^{d}),\norm{\psi}_\infty\le 1}\int_{\R^d\times \{\abs{z}\le R\}} \tilde \psi \chi \bdx,
\end{align*}
and analogously for $\Vb_R$ (cf. \eqref{eq:Vb_def}).

It suffices to show:
\begin{enumerate}[label= Step \arabic*:]
\item There exists $C<\infty$ such that for any $R\ge C\V^{\frac 1{d+1}}$, there holds
\begin{align}
\V-\V_R \lesssim \E^{\frac {d+1}{d}}.\label{v1}
\end{align}
\item For any $R\geq 0$ there holds
\begin{align}
\V_R-\Vt_R\ls\E^{\frac {d+1}{d}}.\label{v2}
\end{align}
\item There exists $C<\infty$ such that for any $R\sim \V^{\frac 1{d+1}}$, there holds
\begin{align}
\Vt_R-\Vb_R\le \frac 12 \V+ C\E^{\frac {d+1}{d}}.\label{v3}
\end{align}
\end{enumerate}
Indeed, choosing $C$ from Step $1$ and $R=C\V^{\frac{1}{d+1}}$, it follows that
\begin{align*}
  \V\overset{\eqref{v1}}\leq \V_R+C\E^{\frac {d+1}{d}}\overset{\eqref{v2}}\leq \Vt_R+ C\E^{\frac {d+1}{d}}\overset{\eqref{v3}}\leq \frac 12 \V+\Vb_R+ C\E^{\frac {d+1}{d}}\leq \frac 12 \V+\Vb+ C\E^{\frac {d+1}{d}},
\end{align*}
from which we deduce \eqref{eq:V_Vb}.

\noindent\textbf{Step 1:} Let
\begin{align}\label{eq:chi_R}
  \chi_R\coloneqq \chi \, \one_{\R^d\times\{\abs{z}\leq R\}},
\end{align}
so that $\V_R=\norm{\chi_R}_1$. We denote $\delchi=\chi-\chi_R$ and by $\delV=\delV(R)=\norm{\delchi}_1=\V-\V_R$ the portion of the mass in $\R^d\times\{|z|> R\}$ bounded by $\Gamma$ and the hyperplanes $\R^d\times\set{z=\pm R}$. Note that $\delV$ is monotone and by Fubini's theorem absolutely continuous with
\begin{align}
\frac{\mathrm{d}\delV}{\mathrm{d}R}=\int_{\R^d}\chi(x,R)-\chi(x,-R)\dx\le 0\quad \text{for almost every $R>0$.}
\end{align}
 Without loss of generality, we may assume $\V\gg \E^{\frac {d+1}{d}}$ (since otherwise the bound holds trivially for any $R\geq 0$).

Letting $\E_R=\int_{\Gamma\cap (\R^d\times\set{\abs{z}>R})}1-e_z\cdot n \dS$ denote the excess energy of $\Gamma$ in $\R^d\times\{|z|> R\}$, we have by the divergence theorem for the constant function $e_z$ in a cylindrical domain over $\set{x\in\R^d:\chi(x,z)\neq 0 \text{ for some }\abs{z}>R}$ that
\begin{align}\label{eq:E_R}
  \E_R=|\babla\chi|\llcorner \bra{\R^d\times\set{\abs{z}>R}}+\frac{\mathrm{d}\delV}{\mathrm{d}R},\quad \text{for almost every $R>0$.}
\end{align}
On the other hand for almost every $R>0$ there holds
\begin{align}
  \int_{\R^{d+1}}\abs{\babla \delchi}=|\babla\chi|\llcorner \bra{\R^d\times\set{\abs{z}>R}}-\frac{\mathrm{d}\delV}{\mathrm{d}R}
  \overset{\eqref{eq:E_R}}{=}\E_R-2\frac{\mathrm{d}\delV}{\mathrm{d}R}.
\end{align}
Clearly
\begin{align}
  \E_R-2\frac{d\delV}{dR}\leq \E-2\frac{d\delV}{dR},
\end{align}
so that by the isoperimetric inequality for $\delchi$, we have
\begin{align}
	\delV\le C \bra{\E-2\frac{\mathrm{d}\delV}{\mathrm{d}R}}^{\frac {d+1}{d}}.
\end{align}
Rewriting this as
\begin{align}
	-\frac{\mathrm{d}\delV}{\mathrm{d}R}\ge \frac 1C \delV^{\frac d{d+1}}-\frac 12 \E,
\end{align}
and observing  that $\delV(0)=\V$,
we obtain that as long as $\delV\ge (C\E)^{\frac {d+1}{d}}$, there holds
\begin{align}
	-\frac{\mathrm{d}\delV}{\mathrm{d}R}\ge \frac{1}{2C}\delV^{\frac {d}{d+1}}.
\end{align}
Integration from $0$ to $R$ yields $\V^{\frac 1{d+1}}-\delV^{\frac 1{d+1}}\ge \frac{1}{2C(d+1)}R$ and hence $R\le 2C(d+1)\V^{\frac 1{d+1}}$. By contraposition this implies that for any $R\ge 6C\V^{\frac 1{d+1}}$, there holds  $\delV\le (C\E)^{\frac {d+1}{d}}\ls \E^{\frac {d+1}{d}}$.\\

\noindent\textbf{Step 2:}
The cut-off plays no role in this step (i.e., the argument is the same for any $R\geq 0$), so without loss of generality, we will establish~\eqref{v2} for $R=0$.
Integrating out the $z$-component, we define
\begin{align}\label{eq:gh}
	g(x)\coloneqq \int_{\R}\abs{\chi}(x,z)\dz, \qquad h(x)\coloneqq \int_{\R}\chi(x,z)\dz,
\end{align}
so that $\V$ and $\Vt$ can be represented as
\begin{align}
	\V=\int_{\R^d} g\dx, \qquad \Vt=\int_{\R^d}\abs{h}\dx
\end{align}
and compared by studying the set $F_R:=\set{x\in\R^d\colon \abs{x}<R,\abs{h(x)}<g(x)}$.

Naturally $\abs{h}\le g$. Also because of the structure of $\chi$, there holds
\begin{align*}
\int_{\R}\abs{\partial_z \one_{\Omega_+}(x,z)}\dz\ge 3\qquad \text{on the set }F_R.
\end{align*}
Thus we have
\begin{align}
	\int_{F_R\times \R}\abs{\babla\one_{\Omega_+}}\,\bdx\ge \int_{F_R} \int_\R \abs{\partial_z \one_{\Omega_+}(x,z)}\dz\ge 3\abs{F_R}.\label{mgfr}
\end{align}

In particular
\begin{align}
  |F_R|\leq \frac 13\int_{F_R\times \R}\abs{\babla\one_{\Omega_+}}\,\bdx.\label{F}
\end{align}
As usual, by using the divergence theorem on a cylindrical domain over $F_R$, we have
\begin{align}
\int_{\Gamma\cap (F_R\times \R)}e_z\cdot n\dS=|F_R|.\label{F2}
\end{align}
Combining these facts, we deduce
\begin{align}
	\E&\ge \int_{\Gamma\cap (F_R\times \R)}1-e_z\cdot n\dS\overset{\eqref{F2}}=\int_{F_R\times \R}\abs{\babla \one_{\Omega_+}}\bdx-\abs{F_R}\\
	\overset{\eqref{F}}&\ge \frac 23 \int_{F_R\times \R}\abs{\babla \one_{\Omega_+}}\bdx\overset{\eqref{mgfr}}\geq 2|F_R|.\label{eq:g_h}
\end{align}
As \eqref{eq:g_h} is true independently of $R$, we conclude that the set $F\coloneqq\set{\abs{h}<g}$ has finite measure and $\E\ge \frac 23 \int_{F\times \R}\abs{\babla \one_{\Omega_+}}\bdx$.

For $w=g,h$ we decompose $\nabla w=\nabla^r w+\nabla^s w$, where $\nabla^r w$ is regular with respect to the Lebesgue measure and $\nabla^s w$ denotes the singular part. We will show that
\begin{align}
\int_F\sqrt{\abs{\nabla^r g}^2+1}\dx+\int_F\abs{\nabla^s g}\dx\ls \int_{F\times\R}\abs{\babla \one_{\Omega_+}}\bdx.\label{eq:nabla_g}\\
	\int_F\sqrt{\abs{\nabla^r h}^2+1}\dx+\int_F\abs{\nabla^s h}\dx\ls \int_{F\times\R}\abs{\babla \one_{\Omega_+}}\bdx,\label{eq:nabla_h}	
\end{align}
which together with the isoperimetric inequality on  $\set{(x,z):\abs{h}(x)<z<g(x)}$ implies
\begin{align}
	\V-\Vt&=\int_{\R^d}(g-\abs{h})\dx\\
&\ls \bra{\int_{F}\sqrt{\abs{\nabla^r g}^2+1}\dx+\int_{F}\abs{\nabla^s g}+\int_{F}\sqrt{\abs{\nabla^r h}^2+1}\dx+\int_{F}\abs{\nabla^s h}}^{\frac {d+1}{d}}\\
\overset{\eqref{eq:nabla_g},\eqref{eq:nabla_h}}&\ls \bra{\int_{F}\int_{\R}\abs{\babla \one_{\Omega_+}}\dz\dx }^{\frac {d+1}{d}}\\
	\overset{\eqref{eq:g_h}}&{\ls} \E^{\frac {d+1}{d}}.
\end{align}

The proof of \eqref{eq:nabla_g} and \eqref{eq:nabla_h} are the same, and we will show \eqref{eq:nabla_g}. Notice that we can write distributionally
\begin{align}
	\nabla g(x)=\int_{\R} \sigma \nabla \one_{\Omega_+}\dz,
\end{align}
where $\sigma=\pm 1$ on $(-\infty,0)$ and $(0,\infty)$, respectively and
\begin{align}
	1=\int_{\R} \partial_z \one_{\Omega_+}\dz.
\end{align}
Thus, for any $\xi\in C^\infty_c(\R^d;\R^d)$, $\zeta\in C^\infty_c(\R^d;\R)$, there holds
\begin{align}
	\int_{\R^d}\xi\cdot \nabla g+\zeta \dx=\int_{\R^{d+1}}\bra{\sigma\xi\cdot \nabla\one_{\Omega_+}+\zeta\partial_z\one_{\Omega_+}}\bdx\le \int_{\R^{d+1}}\sqrt{\abs{\xi}^2+\zeta^2}\abs{\babla \one_{\Omega_+}}.
\end{align}
Taking the supremum over $\xi,\zeta$ with $|\xi|^2+\zeta^2\leq 1$  yields the bound in  \eqref{eq:nabla_g}.

\noindent\textbf{Step 3:} We recall the definition of $\chi_R$ from \eqref{eq:chi_R} and denote $\Gamma_R=\Gamma\cap \bra{\R^d\times\set{\abs{z}\le R}}$. We claim that to establish \eqref{v3}, it suffices to show
\begin{align}\label{eq:V_claim}
		\abs{\int_{\R^{d+1}}\chi_R(\tilde \psi-\bar \psi)\bdx}&\le \bra{\V_R\int_{\Gamma_R} \abs{z}|n^\prime|\dS }^{\frac 12}.
\end{align}
(Recall our notation $n^\prime=n-(n\cdot e_z)e_z$.) Indeed, observe that
\begin{align}
		\int_{\Gamma_R} \abs{z}|n^\prime|\dS&=\int_{\Gamma_R} (\abs{z}\wedge R)|n^\prime|\dS\\
		&\le \bra{\int_{\Gamma} (\abs{z}\wedge R)^2\dS\int_{\Gamma}|n^\prime|^2\dS}^{\frac 12}\\
		\overset{\eqref{eq:n2_est},\eqref{eq:zR}}&{\le} R^{\frac 12}\V^{\frac 12}\E^{\frac 12}+R\E,
\end{align}
where we used Lemma~\ref{l:EV_interpolation}. Using $R\sim \V^{\frac 1{d+1}}$, we arrive at
\begin{align}\label{eq:zn_est}
		\int_{\Gamma_R} \abs{z}|n^\prime|\dS\ls \V^{\frac {d+2}{2(d+1)}}\E^{\frac 12}+\V^{\frac 1{d+1}}\E.
\end{align}		
Inserting \eqref{eq:zn_est} into \eqref{eq:V_claim} (and using  $\V_R\le \V$) gives
\begin{align*}
\abs{\int_{\R^{d+1}}\chi_R(\tilde \psi-\bar \psi)\bdx}\ls \V^{\frac {3d+4}{4(d+1)}}\E^{\frac 14}+\V^{\frac {d+2}{2(d+1)}}\E^{\frac 12},
\end{align*}
so that Young's inequality and considering the supremum over $\psi$ leads to \eqref{v3}.

It hence remains only to establish \eqref{eq:V_claim}. The idea is to write $\tilde \psi -\bar \psi=(\Id-\Po)\psi$, where $\Po$ is the convolution with the Poisson kernel
\begin{align}
	P(x,z)=C\frac{\abs{z}}{\bra{\abs{x}^2+z^2}^{\frac{d+1}{2}}},\label{P}
\end{align}
and then shift $(\Id-\Po)$ onto $\chi_R$:
	\begin{align}
		\int_{\R^{d+1}}\chi_R(\tilde \psi-\bar \psi)&\bdx=\int_{\R^{d+1}} \chi_R(y,z)\psi(y)\dd y\,\dz-\int_{\R^{d+1}}\bra{\int_{\R^d}\chi_R(x,z)P(x-y,z)\psi(y)\dd y}\dx\,\dz\\
		&=\int_{\R^{d+1}} \chi_R(y,z)\psi(y)\dd y\,\dz-\int_{\R^d} \psi(y)\bra{\int_\R\underbrace{\int_{\R^d}P(y-x,z)\chi_R(x,z)\dx}_{=:\chi_{R,z}(y)} \dd z}\dd y\\
		&=\int_\R \int_{\R^d} \psi(y) \bra{\chi_R(y,z)-\chi_{R,z}(y)}\dd y\dz\\
		&\le \int_\R \int_{\R^d} \abs{\chi_R(y,z)-\chi_{R,z}(y)}\dd y\dz
	\end{align}
	For fixed $M>0$, using that
\begin{align}\int_{\R^d} P(\eta,z)\dd\eta=1 \text{ for all $z$},\label{pone}
\end{align}
we estimate
	\begin{align}
		\int_{\R^d}&\abs{\chi_R(y,z)-\chi_{R,z}(y)}\dd y= \int_{\R^d} \abs{\int_{\R^d}P(\eta,z)\chi_R(y,z)-\chi_R(y+\eta,z)\dd \eta}\dd y \\
		&\le \int_{\R^d}\int_{\set{\abs{\eta}\le M\abs{z}}}P(\eta,z)\abs{\chi_R(y,z)-\chi_R(y+\eta,z)}\dd \eta\dd y\\
		&\quad+\int_{\R^d}\int_{\set{\abs{\eta}> M\abs{z}}}P(\eta,z)\abs{\chi_R(y,z)-\chi_R(y+\eta,z)}\dd \eta\dd y\\
		&\le\int_{\set{\abs{\eta}\le M\abs{z}}}\bra{P(\eta,z) \int_{\R^d}\int_0^1\abs{\nabla\chi_R(y+s\eta,z)\cdot \eta}\dd s\,\dd y}\dd \eta\\
		&\quad+2\int_{\R^d}\abs{\chi_R}(x,z)\dx\int_{\set{\abs{\eta}> M\abs{z}}}P(\eta,z)\dd \eta\\
		\overset{\eqref{P}}&\ls \int_{\set{\abs{\eta}\le M\abs{z}}}\int_0^1\int_{\R^d}\abs{\nabla \chi_R(y+s\eta,z)}\dd y\dd s
\abs{\eta} P(\eta,z) \dd \eta\\
&\quad+\int_{\R^d}\abs{\chi_R}(x,z)\dx\int_{\set{\abs{\eta}> M\abs{z}}}\frac{\abs{z}}{\bra{\abs{\eta}^2+z^2}^{\frac {d+1}2}}\dd \eta\\
		\overset{\eqref{pone}}&\ls M \abs{z}\int_{\R^d}\abs{\nabla \chi_R}(x,z)\dx +\frac 1M\int_{\R^d}\abs{\chi_R}(x,z)\dx.
	\end{align}
	Integration over $z$ and optimization in $M$ yields \eqref{eq:V_claim}, where we have used \[\nabla \chi_R=n'\abs{\babla \chi} \llcorner\bra{\R^d\cap\set{\abs{z}\le R}}.\]

\end{proof}

\begin{proof}[Proof of Lemma~\ref{l:EV_interpolation}]
Estimate \eqref{eq:n2_est} follows from
	\begin{align}
		|n^\prime|^2=1-(e_z\cdot n)^2=(1+e_z\cdot n)(1-e_z\cdot n)\le 2(1-e_z\cdot n),\label{diffsquares}
	\end{align}
where  we used the fact that $|e_z\cdot n|\le 1$.
	For the proof of \eqref{eq:zR} we compute
	\begin{align}
		\int_{\Gamma} \abs{z}e_z\cdot n\dS&=\int_{\Gamma}\abs{z}e_z\cdot n\dS-\int_{\R^d\times\set{0}}\abs{z}e_z\cdot e_z\dS=-\int_{\R^{d+1}}\chi\Div\bra{\abs{z}e_z}\bdx\\
&=\int_{\R^{d+1}}\abs{\chi}\bdx=\V,
	\end{align}
	and thus
	\begin{align}
		\int_{\Gamma}(\abs{z}\wedge R)^2\dS&\le \int_{\Gamma}(\abs{z}\wedge R)^2 e_z\cdot n\dS+\int_{\Gamma}(\abs{z}\wedge R)^2(1-e_z\cdot n)\dS\\
		&\le R\int_{\Gamma}\abs{z}e_z\cdot n\dS+R^2\int_{\Gamma}1-e_z\cdot n\dS\\
		&= R\V+R^2\E.\label{eq:ddt10}
	\end{align}
\end{proof}

\begin{proof}[Proof of Lemma~\ref{l:preproc_init}]
The idea is to use $\bar u$ as a test function in the weak equation~\eqref{eq:weak_dyn}:
\begin{align}
\bra{\int_{\R^{d+1}} \chi\bar u\bdx}(T)-\bra{\int_{\R^{d+1}}\chi \bar u \bdx}(0)
&=\int_0^T\bra{\int_{\R^{d+1}}\chi\partial_t \bar u \bdx-\int_{\R^{d+1}} \babla f \cdot\babla \bar u \bdx}\dt\\
&=\int_0^T\bra{\int_{\R^{d+1}}\chi\partial_t \bar u \bdx+2\int_{\R^d} f\partial_z\bar u\dx}\dt.\label{ddt0}
\end{align}

Because $\bar u$ is not an admissible test function (it is merely continuous across $\R^d\times \set{z=0}$ and does not have compact support in space), this requires an approximation argument.
Extending \eqref{eq:weak_dyn} to test with functions that are smooth on $\R^d\times \set{z<0}$ and $\R^d\times \set{z>0}$ and continuous across $\R^d\times \set{z=0}$ is straightforward.
For the integrability/support, we
take a smooth cut-off function $\eta_R\in C^\infty_c(B_{R+1}(0))$ with $\eta_R\equiv 1$ on $B_R(0)$ and $\norm{\eta_R}_\infty+\norm{\nabla \eta_R}_\infty\ls 1$ and  test  \eqref{eq:weak_dyn} by $\varphi_R=\eta_R\bar u$ to obtain
\begin{align}
\lefteqn{\bra{\int_{\R^{d+1}} \chi\varphi_R\bdx}(T)-\bra{\int_{\R^{d+1}}\chi \varphi_R \bdx}(0)}\notag\\
&=\int_0^T\bra{\int_{\R^{d+1}}\chi\partial_t \varphi_R \bdx-\int_{\R^{d+1}} \babla f \cdot\babla \varphi_R \bdx}\dt\\
&=\int_0^T\bra{\int_{\R^{d+1}}\chi\eta_R\partial_t \bar{u} \bdx-\int_{\R^{d+1}} \eta_R\babla f \cdot\babla \bar u \bdx-\int_{\R^{d+1}} \bar u\babla f \cdot \babla\eta_R\bdx}\dt.\label{deedeetee}
\end{align}
To pass to the limit in the various terms, we will argue using regularity and decay.

First note that by $L^1$-continuity {(cf. Definition \ref{def:weaksol})}, $\chi$ is in $L^\infty(0,T;L^1(\R^{d+1}))$.
We turn now to a closer examination of $u=G\ast\psi$ and $\bar{u}{(t)}=P\ast G{(T-t)}\ast\psi$, where $G$ is the kernel from \eqref{eq:profile} and $P$ is the Poisson kernel~\eqref{P}.
For the second term on the left-hand side of \eqref{deedeetee}, we observe that $\psi\in C^\infty_c(\R^d)$ and $G \in L^1(\R^d)$ for $t<T$ implies $u(0)\in L^\infty(\R^d)$ and hence, by $\bar{u}=P\ast u$, also $\bar{u}(0)\in L^\infty(\R^{d+1})$. Similarly for the first term on the right-hand side of \eqref{deedeetee}, notice that $\abs{\nabla}\Delta \psi\in L^\infty(\R^d)$ and $G$ is uniformly bounded in $L^1$ whence  $\abs{\nabla}\Delta u={G(T-t)\ast\abs{\nabla}\Delta \psi}$, is in $L^\infty(0,T;L^\infty(\R^{d}))$. Thus, we argue for and exploit that $\partial_t \bar{u}$, as the harmonic extension of $\abs{\nabla}\Delta u$, is in $L^1(0,T;L^\infty(\R^{d+1}))\subseteq L^\infty(0,T;L^\infty(\R^{d+1}))$.

To pass to the limit in the second term on the right-hand side of \eqref{deedeetee}, we deduce from $u\in L^2(0,T;H^1(\R^d))\subseteq L^2(0,T;H^{\frac{1}{2}}(\R^d))$ that $\babla \bar u\in L^2(0,T;L^2(\R^{d+1}))$.

Finally for the last term on the right-hand side of \eqref{deedeetee}, we will use the decay of $\bar{u}$:
\begin{align}\label{decaybaru}
  \abs{\bar{u}(\bx)}\lesssim \frac{1}{\abs{\bx}^d} \quad\text{for }\;\abs{\bx}\gg 1,
\end{align}
cf.\ Lemma~\ref{l:decaybaru}.
The decay \eqref{decaybaru} together with the $L^2$ control of $\nabla f$ and $L^\infty$ control on $\nabla \eta_R$ allows one to conclude that the last  term on the right-hand side of \eqref{deedeetee} vanishes in the limit and we may pass from \eqref{deedeetee} to  \eqref{ddt0}.

As explained above the statement of the lemma, we define $v= -\abs{\nabla}u$ on the flat interface $\R^d\times \set{z=0}$, i.e., $v=\partial_z \bar u$, and introduce $\tilde v$ as the constant continuation in $z$-direction of $v$.  To establish~\eqref{eq:ddt}, it suffices to argue that for almost every time there holds
\begin{align}
  \int_{\R^d} f\partial_z\bar u\dx=\int_{\R^{d+1}} \chi\tilde v \partial_z f\;\bdx-\int_{\Gamma}(1-e_z\cdot n)(\babla\tilde v\cdot n)\dS-\frac{1}{2}\int_{\R^{d+1}}\chi\,\partial_t \tilde u\bdx.\label{nowthis}
\end{align}
We record for reference below that $\babla (\vt e_z)=e_z\otimes \babla \tilde v$ and
\begin{align}
  \Div \bra{\tilde ve_z}=0\label{divnull}
\end{align}
and hence also
	\begin{align}
		\Div_{\tan}\bra{\tilde ve_z}\coloneqq(\Id-n\otimes n)\colon \!\babla \bra{\tilde ve_z}=-n\cdot \babla \bra{\tilde ve_z} n=-(e_z\cdot n)(\babla \tilde v\cdot n).\label{distdiv}
	\end{align}
By the divergence theorem and the distributional definition~\eqref{eq:weak_H} of the mean curvature on $\Gamma$, there holds
	\begin{align}
		\int_{\R^d} f\partial_z\bar u\dx&=\int_{\R^d}f ve_z\cdot e_z\dx=-\int_{\R^{d+1}}\one_{\R^d\times\set{z>0}}\Div(f\tilde ve_z)\bdx\\
		&=\int_{\R^{d+1}}\chi\Div(f\tilde ve_z)\bdx-\int_{\Gamma}\Div_{\tan}\bra{\tilde ve_z}\dS. \label{eq:ddt2}
	\end{align}
Using~\eqref{divnull}, the first integral on the right-hand side simplifies to
\begin{align}
  \int_{\R^{d+1}}\chi\Div(f\tilde ve_z)\bdx=\int_{\R^{d+1}} \chi\tilde v \partial_z f \;\bdx.\label{this2}
\end{align}
Substituting~\eqref{this2} into~\eqref{eq:ddt2} it suffices for~\eqref{nowthis} to show
\begin{align}
 \int_{\Gamma}\Div_{\tan}\bra{\tilde ve_z}\dS =\int_{\Gamma}(1-e_z\cdot n)(\babla\tilde v\cdot n)\dS+\frac{1}{2}\int_{\R^{d+1}}\chi\,\partial_t \tilde u\bdx.\label{this3}
\end{align}
To this end, we compute
	\begin{align}
		\int_{\Gamma}\Div_{\tan}\bra{\tilde ve_z}\dS\overset{\eqref{distdiv}}&=-\int_{\Gamma}(e_z\cdot n)(\babla\tilde v\cdot n)\dS\notag\\
&=\int_{\Gamma}(1-e_z\cdot n)(\babla \tilde v\cdot n)\dS-\int_{\Gamma} \babla\tilde v\cdot n\dS.\label{H}
	\end{align}
	Finally we transform the second integral on the right-hand side as
	\begin{align}
		\int_{\Gamma}\babla \tilde v\cdot n\dS&=\int_{\Gamma}\babla \tilde v\cdot n\dS-\int_{\R^d}\babla \tilde v\cdot e_z\dS=-\int_{\R^{d+1}}\chi\Div(\babla \tilde v )\bdx\\
		&=\int_{\R^{d+1}}\chi\abs{\nabla}\Delta \tilde u\bdx\\
		&=-\frac{1}{2}\int_{\R^{d+1}}\chi\partial_t\tilde u\bdx. \label{eq:ddt3}
	\end{align}	
	Substituting  \eqref{eq:ddt3} into \eqref{H} demonstrates~\eqref{this3}. This concludes the proof of \eqref{eq:ddt}.

We remark for reference below that we will regularly use the maximum principle in the form
\begin{align*}
\norm{{\ub}}_\infty  \leq \norm{{\uu}}_\infty\;(\overset{\eqref{eq:zeta_psi}}\ls \norm{\psi}_\infty)
\end{align*}
and also for higher derivatives; see Lemma~\ref{l:linear_decay}.

	We start by estimating $A_1$. \sout{Recalling the notation $u$ for the restriction of $\bar{u}$ to $\R^d$ (cf.\ \eqref{eq:dual_lin1}),} We estimate
	\begin{align}
		A_1&\le \norm{v}_\infty \int_{\R^{d+1}}\abs{\chi}\abs{\babla f}\bdx\le \norm{v}_\infty\bra{\int_{\R^{d+1}}\abs{\chi}^2\bdx\int_{\R^{d+1}}\abs{\babla f}^2\bdx}^{\frac 12}\\
		&\le \norm{v}_{\infty}\V^{\frac 12}D^{\frac 12}.\label{eq:ddt5}
	\end{align}
	An application of \eqref{eq:v_decay1} yields \eqref{eq:A1_est}. Inserting \eqref{eq:v_decay2} in
	\begin{align}
		A_2\ls \norm{\nabla v}_{\infty}\int_{\Gamma}1-e_z\cdot n\dS\ls \norm{\nabla v}_\infty \E\label{eq:ddt6}
	\end{align}
	delivers \eqref{eq:A2_est}.

Finally, we turn to $A_3$, which we rewrite as
\begin{align}
		A_3&=\abs{\int_{\R^{d+1}}\chi(\partial_t\bar u-\partial_t \tilde u)\bdx}=2\abs{\int_{\R^{d+1}}\chi(\Delta \tilde v-\Delta \bar v)\bdx}.\notag
	\end{align}
Using the divergence theorem and $\nabla\tilde v(x,0)=\nabla \bar v(x,0)=\nabla v(x)$, we reformulate the right-hand integral as
	\begin{align}
		\abs{\int_{\R^{d+1}}\chi(\Delta \tilde v-\Delta \bar v)\bdx}
		&=\abs{\int_{\Gamma}n^\prime\cdot \nabla(\tilde v-\bar v)\dS},\label{eq:ddt7}
	\end{align}
	and estimate for $R>0$ via
	\begin{align}
		\lefteqn{\abs{\int_{\Gamma}n^\prime\cdot \nabla (\tilde v-\bar v)\dS}}\\
&\ls\bra{R^{-1}\norm{\nabla v}_{L^\infty(\R^{d})}+\norm{\abs{z}^{-1}\nabla(\tilde v-\bar v)}_{L^\infty(\R^{d+1})}}\int_{\Gamma}|n^\prime|(\abs{z}\wedge R)\dS\\
&\ls\bra{R^{-1}\norm{\nabla v}_{L^\infty(\R^{d})}+\norm{\partial_z\nabla v}_{L^\infty(\R^{d})}}\int_{\Gamma}|n^\prime|(\abs{z}\wedge R)\dS,  \label{eq:ddt8}
	\end{align}
where we have used
$\norm{\abs{z}^{-1}\nabla(\tilde v-\bar v)}_{\infty}\le \norm{\partial_z\nabla \bar v}_\infty$
and we have applied the maximum principle (more than once).

Next, using Lemma~\ref{l:EV_interpolation}, we estimate
	\begin{align}
		\int_{\Gamma}|n^\prime|(\abs{z}\wedge R)\dS\le \bra{\int_{\Gamma}|n^\prime|^2\int_{\Gamma}(\abs{z}\wedge R)^2}^\frac 12\le \E^{\frac 12}\bra{R\V+R^2\E}^{\frac 12}.\label{eq:ddt9}
	\end{align}
	We insert \eqref{eq:ddt9} in \eqref{eq:ddt8} to obtain
	\begin{align}
		\abs{\int_{\Gamma}n^\prime\cdot \nabla^\prime(\tilde v-\bar v)\dS}&\ls \bra{R^{-\frac 12}\norm{\nabla v}_\infty+R^{\frac 12}\norm{\partial_z\nabla v}_\infty}\E^{\frac 12}\bra{\V+R\E}^{\frac 12},\label{eq:ddt11}
	\end{align}
which we in turn substitute into \eqref{eq:ddt7} to deduce
	\begin{align}
		\abs{A_3}\ls \bra{R^{-\frac 12}\norm{\nabla v}_\infty+R^{\frac 12}\norm{\partial_z \nabla v}_\infty}\E^{\frac 12}\bra{\V+R\E}^{\frac 12}.\label{eq:ddt12}
	\end{align}
	The choice $R=(T-t)^{\frac 13}$ and the use of \eqref{eq:v_decay2} and \eqref{eq:v_decay3} finishes the proof of \eqref{eq:A3_est}.
\end{proof}

\section{Entering the graph setting}\label{S:graph}

In this section we prove Proposition~\ref{prop:2}; in particular, $d=2$. We will, in the spirit of Notation~\ref{not:normal_bold}, write balls in dimension $2$ in regular font and in dimension $3$ boldface, i.e. for $x_0\in \R^{2}$ and $\bx_0\in \R^3$, we denote
\begin{align*}
B_\rho(x_0)&=\set{x\in \R^2:\abs{x-x_0}<\rho}, \qquad \bB_\rho(\bx_0)=\set{\bx\in \R^{3}:\abs{\bx-\bx_0}<\rho}.
\end{align*}

For the proof of Proposition~\ref{prop:2}, we need two ingredients.
First, we need a version of Allard's regularity theorem. We define the tilt excess with respect to $\R^2$ as
\begin{align}\label{eq:tilt_exc}
	E(\bx,\rho)=\rho^{-2}\int_{\bB_\rho(\bx)\cap \Gamma}1-(n\cdot e_z)^2\dS.
\end{align}

We rely on Allard's Regularity Theorem \cite[Section 8]{All} in the form given by Simon~\cite[Theorem 23.1]{Sim}, which can be stated in our BV setting in the form:
\begin{theorem}[Allard's Regularity Theorem]\label{thm:allard}
	For any $p>2$ and $\alpha\in (0,1)$ there are $\eps_\alpha>0$, $\gamma\in (0,1)$, and $C<\infty$ with the following property. Let $\Omega_+$ be a set with locally finite perimeter, $\Gamma=\sppt \babla \one_{\Omega_+}$, and the generalized scalar mean curvature $H$ of $\Gamma$ given by a measurable function on $\Gamma$. If for $\bx=(x,z)\in \Gamma$ the three conditions
	\begin{enumerate}[label=(\roman*)]
		\item $\abs{\babla \one_{\Omega_+}}(\bB_\rho(\bx))\le 2(1-\alpha)\pi\rho^2$, \label{it:allard1}
		\item $E(\bx,\rho)\le \eps_\alpha$, \label{it:allard2}
		\item $\rho^{2\bra{1-\frac 2p}}\norm{H}^2_{L^p(\bB_\rho(\bx)\cap \Gamma)}\le \eps_\alpha^2$, \label{it:allard3}
	\end{enumerate}
	hold, then $\bB_{\gamma\rho}(\bx)\cap \Gamma$ is the graph of a Lipschitz function $h$ over $B_{\gamma\rho}(x)$ satisfying the estimate
	\begin{align}
		\norm{\nabla h}_{L^\infty(B_{\gamma\rho}(x))}\le C\bra{E(\bx,\rho)^{\frac 12}+\rho^{1-\frac 2p}\norm{H}_{L^p(\bB_\rho(\bx)\cap \Gamma)}}.\label{gradhest}
	\end{align}
\end{theorem}

Equivalently, one can work with the newer version \cite[Theorem 5.5.2]{Simon_new} and straightforward modifications of our proof.

To establish~\ref{it:allard3} and control the second term on the right-hand side of \eqref{gradhest}, we need $L^p$ control of the mean curvature for $p>2=d$, which we obtain in the following form.
\begin{lemma}\label{l:H_trace}
	Assume that $\E D^2<\infty$. Then the distributional mean curvature is given by a locally integrable function $H\in L^4(\Gamma)$ and
	\begin{align}
		\norm{H}_{L^4(\Gamma)}\ls \bra{1+(\E D^2)^{\frac{1}{20}}}D^{\frac 12}.\label{hbound}
	\end{align}
\end{lemma}
We will deduce this estimate by combining a trace estimate of Meyers and Ziemer with a monotonicity formula of Schätzle; see Subsection~\ref{ss:lemma41} below.

\begin{proof}[Proof of Proposition~\ref{prop:2}]
 	Fix  $p=4$, $\alpha\in (0,\frac 12)$, and $\bx_0=(x_0,z_0)\in \Gamma$. We will check that there exists $\rho>0$ such that the conditions of Theorem~\ref{thm:allard} hold. Note  for reference below that the tilt excess is controlled by the energy as in the proof of Lemma~\ref{l:EV_interpolation} via
 	\begin{align}
 		E(\bx_0,\rho)
 		=\rho^{-2}\int_{\bB_\rho(\bx_0)\cap \Gamma}1-(e_z\cdot n)^2\dS
 		\overset{\eqref{diffsquares}}\le 2\rho^{-2}\E.\label{eq:tilt_est}
 	\end{align}
  For condition~\ref{it:allard1} we calculate as for \eqref{eq:E_R}, using the divergence theorem on the intersection of an infinite cylinder of radius $\rho>0$ with $\Gamma$ and $\{z=0\}$, that
	\begin{align}
		\int_{\Gamma\cap \set{(x,z):\abs{x-x_0}<\rho}} e_z\cdot n\dS= \pi \rho^2.\label{rhosquare}
	\end{align}
 Consequently for any $\rho>0$, by expanding the ball to a cylinder, there holds
	\begin{align}
		\abs{\babla \one_{\Omega_+}}(\bB_\rho(\bx_0))&
\le\int_{\Gamma\cap \set{(x,z):\abs{x-x_0}<\rho}}1\dS
		 =\int_{\Gamma\cap \set{(x,z):\abs{x}<\rho}}e_z\cdot n\dS+\E\\
\overset{\eqref{rhosquare}}&=\pi \rho^2+\E.\label{eq:haus1}
	\end{align}
 	From here we read off that condition~\ref{it:allard1} holds as long as $\rho$ is large enough so that
	\begin{align}\label{eq:allard1}
		\rho^{-2}\E\le (1-2\alpha)\pi,\quad\text{i.e.,}\quad \rho\geq \sqrt{\frac{\E}{(1-2\alpha)\pi}}.
	\end{align} 	
 	At the  same time, the estimate \eqref{eq:tilt_est} guarantees condition \ref{it:allard2} as long as
 	\begin{align}\label{eq:allard2}
 		\rho^{-2}\E\le \frac 12 \eps_\alpha,\quad\text{i.e.,}\quad \rho\geq \sqrt{\frac{2\E}{\eps_\alpha}}.
 	\end{align}
 Finally, restricting $\eps_1$ from Proposition~\ref{prop:2} to $\eps_1\le 1$, Lemma~\ref{l:H_trace} yields \ref{it:allard3} for all $\rho$ small enough so that
 	\begin{align}\label{eq:allard3}
 		2C\rho D\le \eps_\alpha^2,
 	\end{align}
 	where $C$ is the implicit constant in \eqref{hbound}.
 	Choosing
 \begin{align}
 \rho=\bra{\frac{\E}{D}}^{\frac 13}\label{rhob}
 \end{align}
  and $\E D^2\leq \eps_1$ for $\eps_1>0$ small enough delivers \eqref{eq:allard1}, \eqref{eq:allard2}, and \eqref{eq:allard3}.

 For this radius $\rho$, the surface $\Gamma\cap \bB_{\gamma\rho}(\bx_0)$ is a graph over the disk $B_{\gamma\rho}(x_0)$. Since $\bx_0\in \Gamma$ was arbitrary, $\Gamma$ is locally a graph over $\R^2$. We consider the set
\begin{align*}
 	M=\set{x\in \R^2: \#(\Gamma\cap (\set{x}\times \R))=1}.
\end{align*} 	
Note that $\#(\Gamma\cap (\set{x}\times \R))\ge 1$ everywhere. Since $\E<\infty$, the set $M$ is not empty. Because of the local graph property, $M$ and its complement $M^c$ are open, and hence $M=\R^2$ and $\Gamma$ is given by a global Lipschitz graph function $h$.
 Finally, the bound from Theorem~\ref{thm:allard} in combination with \eqref{eq:tilt_est} and Lemma~\ref{l:H_trace} yields
 	\begin{align}
 		\norm{\nabla h}_\infty\ls \rho^{-1}\E^{\frac 12}+\rho^{\frac 12}D^{\frac 12}\overset{\eqref{rhob}}=2 (\E D^2)^{\frac 16}.
 	\end{align}
\end{proof}

\subsection{Proof of Lemma~\ref{l:H_trace}}\label{ss:lemma41}

We capitalize on the fact that $H$ is the trace of $f$ and employ the following Meyers-Ziemer trace estimate, contained in
 {\cite[Theorem 4.7]{MZ}}; see also {\cite[Theorem 5.12.4]{Zie}} and \cite[p. 386]{Sch}. (We again formulate the results in the BV setting.)
\begin{lemma}[Meyers-Ziemer trace estimate]\label{l:ziemer}
Let $\Omega_+$ be a set with locally finite perimeter and $\Gamma=\sppt \babla \one_{\Omega_+}$. Let
\begin{align}
	M(\Omega_+)\coloneqq\sup_{\bx\in \R^{3}, \rho>0}\rho^{-2}\abs{\babla \one_{\Omega_+}}(\bB_\rho(\bx)).
\end{align}
	If $M(\Omega_+)<\infty$, then for all $\varphi\in C^1_0(\R^{3})$ there holds
	\begin{align}
		\abs{\int_\Gamma \varphi \dS}\ls M(\Omega_+)\norm{\nabla \varphi}_{1}.
	\end{align}

\end{lemma}
To obtain a uniform bound on $M(\Omega_+)$, we will use the following slight adaption of a monotonicity formula of Sch\"atzle, cf.\ {\cite[Lemma 2.1]{Sch}, whose proof follows \cite[Section 17]{Sim}.
\begin{lemma}[Schätzle's monotonicity formula]\label{l:schaetzle}
	There is a constant $C<\infty$ with the following property. Let $\Omega_+$ be a set with locally finite perimeter and $\Gamma=\sppt \babla \one_{\Omega_+}$, and let $f\in \Hd^1(\R^{3})$ satisfy \eqref{eq:weak_H}. Then, for any $\bx\in \R^3$, the function

	\begin{align}
		\rho\mapsto \rho^{-2}\abs{\babla \one_{\Omega_+}}( \bB_\rho(\bx))+C\rho^{\frac 12}\norm{\nabla f}_2
	\end{align}
	is nondecreasing.
\end{lemma}
\begin{proof}[Proof of Lemma~\ref{l:H_trace}]
	Using density we apply  the Meyers-Ziemer estimate from Lemma~\ref{l:ziemer} with $\varphi=f^4$ to obtain
	\begin{align}
		\norm{f}^4_{L^4(\Gamma)}&\ls M(\Omega_+)\norm{\nabla(f^4)}_1\ls M(\Omega_+)\norm{f^3 \nabla f}_1\ls M(\Omega_+) \norm{\nabla f}_2\norm{f}_6^3\ls M(\Omega_+) \norm{\nabla f}_2^4\\
		&\ls M(\Omega_+) D^2, \label{eq:H_bound}
	\end{align}
	where we used Sobolev embedding in the penultimate step. Since $f$ has a trace in $L^4$ on $\Gamma$, an application of the divergence theorem to the right-hand side of \eqref{eq:weak_H} reveals that $f=H$ on $\Gamma$ in $L^4(\Gamma)$. To deduce~\eqref{hbound} it hence suffices to show
	\begin{align}\label{eq:M_bound}
		M(\Omega_+)\lesssim 1+(\E D^2)^{\frac 15}.
	\end{align}

To this end, we observe that for any fixed $R>0$, Schätzle's monotonicity formula (Lemma~\ref{l:schaetzle}) implies  for $\rho\leq R$ that
	\begin{align}
		\rho^{-2}\abs{\babla \one_{\Omega_+}}( \bB_\rho(\bx))&\le R^{-2}\abs{\babla \one_{\Omega_+}}( \bB_R(\bx))+CR^{\frac 12}D^{\frac 12}\\
		\overset{\eqref{eq:haus1}}&{\le} \pi+R^{-2}\E+CR^{\frac 12}D^{\frac 12}\lesssim 1+(\E D^2)^\frac{1}{5},\label{eq:haus_D}
	\end{align}
where in the last step we have optimized with
	\begin{align}
		R=\bra{\frac{\E}{D^{\frac 12}}}^{\frac 25}.
	\end{align}
It remains to control the supremum over larger radii, which is achieved by observing that
 \eqref{eq:haus1}  for $\rho\ge R$ gives
	\begin{align}\label{eq:haus_E}
		\rho^{-2}\abs{\babla \one_{\Omega_+}}( \bB_\rho(\bx))\le \pi+R^{-2}\E\lesssim 1+(\E D^2)^\frac{1}{5}.
	\end{align}
The combination of \eqref{eq:haus_D} and \eqref{eq:haus_E} establishes~\eqref{eq:M_bound}.
\end{proof}

\section{Relaxation rates in the graph setting}\label{S:relax2}

In this section, we work in the graph setting with small Lipschitz constant and hence,
by the last part of Hypothesis \textbf{(H)}, assume that all quantities are smooth. In particular, equations \eqref{eq:f} and \eqref{eq:Vel} hold pointwise.

Because $\Gamma$ is given as the graph of a function as in \eqref{eq:graph}, the excess mass $\V$ reduces to the $L^1$-norm of the height function
\begin{align}
 \V&=\int_{\R^d}\abs{h}\dx\label{eq:Vh}.
 \end{align}
Moreover, the following geometric quantities can also be expressed in terms of $h$ as
\begin{align}
  n&=\frac{(-\nabla h,1)}{\sqrt{1+\abs{\nabla h}^2}},\label{eq:n}\\
  H&=\Div\frac{\nabla h}{\sqrt{1+\abs{\nabla h}^2}}=\frac{1}{\sqrt{1+\abs{\nabla h}^2}}\bra{\Delta h-\frac{\nabla^2h:(\nabla h\otimes \nabla h)}{1+\abs{\nabla h}^2}},\label{eq:H}\\
  \vel&=\frac{h_t}{\sqrt{1+\abs{\nabla h}^2}}.\label{eq:norm_vel}
\end{align}
Here the left-hand side quantities are only defined on $\Gamma$ and are hence evaluated at $(x,h(x))$ while the right-hand side quantities are defined on $\R^d$ and are evaluated at $x$. We will mildly abuse notation by allowing context to make clear whether $H,V$ take arguments from $\R^d$ or $\Gamma$. Because of the Lipschitz bound \eqref{eq:Lipschitz} on $h$, within the Lipschitz regime we have
\[\dS=\sqrt{1+\abs{\nabla h}^2}\dx\sim \dx\]
and hence we can
compare quantities on $\Gamma$ to quantities on $\R^d$. A direct consequence of this and \eqref{eq:n} is that the energy can be expressed as
\begin{align}
  \E&=\int_{\R^d}\sqrt{1+\abs{\nabla h}^2}-1\dx\label{eq:Eh}.
\end{align}

A second advantage of the Lipschitz setting is that we can deduce nonlinear estimates from the linearized ones and then take advantage of the (nonlinear) gradient flow structure of the system.

Within the Lipschitz regime, the rough plan is as follows: A simple differential inequality (Lemma~\ref{l:ED2}) implies that the solution remains trapped in the Lipschitz regime. The main interpolation estimate (Proposition~\ref{prop:EED}) and the gradient flow structure imply that as long as the $L^1$ norm remains controlled, the natural algebraic decay estimates hold. A duality argument verifies that the $L^1$ control is indeed guaranteed; see Subsection~\ref{ss:dual}. The pieces are assembled via a buckling argument in Subsection~\ref{ss:main}.

\paragraph{Fractional and negative spaces, notation.}
We will work with the spaces $H^s$ for $s=\pm \frac 12, \pm 1$. For $s=\frac 12$, we define $\dot{H}^{\frac{1}{2}}(\R^d)$ and $\dot{H}^{\frac{1}{2}}(\Gamma)$ via trace, i.e., the space consists of traces of $\dot H^1(\R^{d+1})$ functions and the norm is given by the Dirichlet norm of the corresponding harmonic extension (or equivalently the infimum of the Dirichet norm for all extensions).  The negative spaces are defined by duality.

As in the introduction, we use $\abs{\nabla}^\alpha, \alpha\in \R$ in the Fourier multiplier sense, i.e., if $\F$ and $\F^{-1}$ are the Fourier and the inverse Fourier transform, respectively, then

\begin{align*}
\abs{\nabla}^\alpha g=\F^{-1}\bra{\abs{k}^\alpha\F g}.
\end{align*}
Within our small Lipschitz setting, the trace and Fourier definitions are equivalent in the sense that $\norm{g}_{\dot{H}^{\pm 1/2}(\R^d)}\sim \norm{\abs{\nabla}^{\pm 1/2}g}_{2}$ (see e.g. \cite[Proposition A.1]{COW}).
Also the Sobolev norms on $\R^d$ and $\Gamma$ are equivalent; i.e., for a function $w_\Gamma:\Gamma\to \R$ and $w:\R^d\to \R$ defined by $w(x)=w_\Gamma(x,h(x))$, we have
\begin{align}
\norm{ w_\Gamma}_{\dot H^s(\Gamma)}\sim \norm{w}_{H^s(\R^d)}
\end{align}
for $s=\pm \frac 12, \pm 1$. For $s=\frac 12$ this follows from the fact that the transformation $z'=z-h(x)$ maps $\Omega_+$ to the half-space while the the Dirichlet energy of the corresponding extensions remains comparable (since $h$ is Lipschitz). For $s=1$ it is an easy computation. For the negative spaces it follows by duality. We will make regular use of this fact.

We continue to denote norms on the interface explicitly and suppress the domain on $\R^d$, e.g.,
\begin{align*}
 \norm{g}_{\dot H^{-\frac{1}{2}}(\Gamma)} , \qquad \norm{g}_{\dot H^{-\frac{1}{2}}},
\end{align*}
except when $\R^d$ is made explicit for emphasis.

\subsection{Algebraic relationships}\label{ss:alg}

The aim of this subsection is to relate the nonlinear quantities tied to the MS evolution to linear ones (for which we can take advantage of standard interpolation estimates and Fourier techniques). All statements are for general dimension $d\ge 1$ unless otherwise indicated.

As a first easy consequence of the Lipschitz bound, we obtain that $\E$ scales like the Dirichlet energy of $h$, and  $D$ controls a negative norm of the normal velocity $\vel$.  A more elementary proof of \eqref{eq:nablaf_bdry_est} in the case $d=1$ is contained in \cite[Lemma 3.1 and proof of Lemma 4.1]{COW}.

\begin{lemma}\label{l:V_D}
Under the condition $\norm{\nabla h}_\infty\le 1$, there holds
\begin{align}\label{eq:E_sim_L2}
\E\sim\norm{\nabla h}_2^2,
\end{align}
and
\begin{align}\label{eq:nablaf_bdry_est}
\norm{\vel}_{\Hd^{-\frac 12}(\Gamma)}^2\ls \norm{\nabla f_+\cdot n}_{\Hd^{-\frac 12}(\Gamma)}^2+\norm{\nabla f_-\cdot n}_{\Hd^{-\frac 12}(\Gamma)}^2\ls D.
\end{align}
\end{lemma}
\begin{proof}
We deduce \eqref{eq:E_sim_L2} directly from the identity
\begin{align}\label{eq:linearizeE}
\abs{\nabla h}^2=\bra{\sqrt{1+\abs{\nabla h}^2}-1}\bra{\sqrt{1+\abs{\nabla h}^2}+1}.
\end{align}
The first estimate in \eqref{eq:nablaf_bdry_est} follows from \eqref{eq:Vel}. For the second estimate in \eqref{eq:nablaf_bdry_est}, we test with $\varphi\in\dot H^{\frac 12}(\Gamma)$ and use the divergence theorem (and $\belta f_\pm=0$) to estimate
\begin{align*}
	\abs{\int_\Gamma \varphi \nabla f_\pm\cdot n\dd S}=\abs{\int_{\Omega_\pm} \nabla\varphi \cdot\babla f_\pm\dd S}\le \norm{\nabla f}_2\norm{\nabla \varphi}_2.
\end{align*}
Taking the supremum over all $\varphi$ yields the claim.
\end{proof}

\begin{lemma}\label{l:H}
Assume $\norm{\nabla h}_\infty\le 1$. In any dimension, there holds
\begin{align}
\norm{H}_{\Hd^{-1}(\Gamma)}^2&\ls \E,\label{eq:H_bdry_est_E}\\
\norm{H}_{\Hd^{\frac 12}(\Gamma)}^2&\ls D,\label{eq:H_bdry_est_D}\\
\norm{H}_{L^2(\Gamma)}^2&\ls \bra{\E D^2}^{\frac 13}.\label{eq:H_L2}
\end{align}
\end{lemma}
\begin{corollary}
Let $d=2$. Assume $\norm{\nabla h}_\infty\le 1$. Then for any $2\le p\le 4$ there holds
\begin{align}
\norm{H}_{L^p(\Gamma)}&\ls \E^{\frac{4-p}{6p}}D^{\frac{2(p-1)}{3p}}.\label{eq:H_Lp}
\end{align}
\end{corollary}
\begin{proof}
By embedding (comparing to norms on $\R^d$) we have $\norm{H}_{L^4(\Gamma)}\ls \norm{H}_{\Hd^{\frac 12}(\Gamma)}$. Then \eqref{eq:H_Lp} follows by interpolation between this inequality and \eqref{eq:H_L2}, and by using \eqref{eq:H_bdry_est_D}, where we have again used the equivalence of the norms on $\Gamma$ and $\R^d$ and the Fourier representation, for which the interpolation is readily available.
\end{proof}
\begin{proof}[Proof of Lemma~\ref{l:H}]
Estimate \eqref{eq:H_L2} follows by interpolation between \eqref{eq:H_bdry_est_E} and \eqref{eq:H_bdry_est_D}. Estimate \eqref{eq:H_bdry_est_D} is a Sobolev trace inequality for $f$. Inequality \eqref{eq:H_bdry_est_E} follows from testing equation \eqref{eq:H} with an $\Hd^1(\R^d)$ function, integration by parts and an application of the Cauchy-Schwarz inequality (cf. \cite[Lemma 3.1 (3.3)]{COW} for the case $d=1$).
\end{proof}

We next turn to $p$-norm control of the full Hessian of $h$. Our first observation is that in any dimension, one can directly deduce $L^2$ control of the Hessian from \eqref{eq:H} (see \eqref{eq:L2_h_H} below). In $d=1$, this yields control of $\norm{\nabla h}_\infty$ in terms of $\E$ and $D$ (cf. Corollary~\ref{cor:Lip} below and \cite[Lemma 3.2]{COW}). In $d=2$ the norm $\norm{\nabla^2 h}_2$ \emph{just fails} to control $\norm{\nabla h}_\infty$. Although we establish control of $\norm{\nabla h}_\infty$ in $d=2$ by heavier machinery in Proposition~\ref{prop:2}, we can also deduce it in an elementary way from higher $p$-integrability of the Hessian (see Corollary~\ref{cor:Lip}). We also use the $p>2$ integrability in Lemma~\ref{l:D_decay} below to control the time change of the dissipation, which we need for our duality argument for large times.
\begin{lemma}[Control of the Hessian by mean curvature]\label{l:pnorm_D}
Assume
$\norm{\nabla h}_\infty\le 1$. In any dimension, there holds
\begin{align}\label{eq:L2_h_H}
\norm{\nabla^2 h}_2\ls \norm{H}_2,
\end{align}
and there exists $\delta>0$ such that for any $p\in [2,2+\delta)$, there holds
\begin{align}\label{eq:pnorm_D}
\norm{\nabla^2h}_p\lesssim \norm{H}_p.
\end{align}
\end{lemma}

\begin{proof}
It suffices to show \eqref{eq:pnorm_D}.
We begin with the representation \eqref{eq:H} in the form
\begin{align}\label{eq:h_H}
\Delta h=\sqrt{1+\abs{\nabla h}^2}H+\frac{\nabla h\otimes \nabla h}{1+\abs{\nabla h}^2}:\nabla^2h.
\end{align}
According to the Calderon-Zygmund theory,  there exists $M_p<\infty$ such that solutions of $\Delta w=g$ in $\R^d$ satisfy
\begin{align*}
\norm{\nabla^2 w}_p\le M_p\norm{g}_p\qquad\text{and}\qquad M_2=1.
\end{align*}
Applying this to \eqref{eq:h_H} under the assumption that $\norm{\nabla^2 h}_p<\infty$ yields the estimate
\begin{align*}
\norm{\nabla^2 h}_p \le M_p\bra{\left\Vert\sqrt{1+\abs{\nabla h}^2}H\right\Vert_p+\left\Vert\frac{\nabla h\otimes \nabla h}{1+\abs{\nabla h}^2}\right\Vert_\infty \norm{\nabla^2h}_p}.
\end{align*}
Since
\begin{align*}
\norm{\nabla h}_\infty\le 1, \qquad \text{and hence also}\qquad\left\Vert\frac{\nabla h\otimes \nabla h}{1+\abs{\nabla h}^2}\right\Vert_\infty\le \frac 12,
\end{align*}
we obtain
\begin{align}\label{eq:Meyer}
\norm{\nabla^2 h}_p \le M_p\norm{H}_p+\frac 12 M_p\norm{\nabla^2h}_p.
\end{align}
We now use that $\lim_{p\to 2}M_p= 1$, which is a consequence of the Riesz-Thorin interpolation theorem and $M_2=1$. Choosing $p>2$ close enough to $2$, we have $\frac 12 M_p<1$ and we can absorb the second term on the right-hand side of \eqref{eq:Meyer} into the left-hand side to obtain
\begin{align*}
\norm{\nabla^2 h}_p \ls \left\Vert H\right\Vert_p.
\end{align*}
A fixed point argument then implies $\nabla^2 h\in L^p$ as long as $\frac 12 M_p<1$.

We remark that the proof simplifies in $d=1$, since
\eqref{eq:H} takes the form
\begin{align*}
H&=\frac{d}{dx}\frac{h_x}{\sqrt{1+h_x^2}}=\frac{h_{xx}}{(\sqrt{1+h_x^2})^3}.
\end{align*}
\end{proof}

We now confirm that our integral estimates on the Hessian and previous algebraic relationships can be converted into a Lipschitz bound in $d=2$ and $d=1$.
\begin{corollary}\label{cor:Lip}
Let $d=2$ or $d=1$ and assume $\norm{\nabla h}_\infty\le 1$. Then
\begin{align}\label{eq:Lip_control}
 \norm{\nabla h}_\infty\ls \bra{\E^{3-d} D^d}^{\frac 16}.
\end{align}
\end{corollary}
\begin{proof}

For $d=1$ this is contained in \cite[Lemma 3.2]{COW} and can be deduced from the interpolation estimate
\begin{align*}
\norm{h_x}_\infty\ls \norm{h_x}_2^{\frac 12}\norm{h_{xx}}_2^{\frac 12}\overset{\eqref{eq:E_sim_L2},\eqref{eq:L2_h_H}}{\ls} \E^{\frac 14}\norm{H}_2^{\frac 12}\overset{\eqref{eq:H_L2}}{\ls} \E^{\frac 13}D^{\frac 16}.
\end{align*}
For $d=2$, let $p>2$ be the exponent from Lemma~\ref{l:pnorm_D}. Then
\begin{align*}
\norm{\nabla h}_\infty\ls \norm{\nabla h}_2^{\frac{p-2}{2(p-1)}}\norm{\nabla^2h}_p^{\frac{p}{2(p-1)}}
\end{align*}
(cf.\ Lemma~\ref{l:GNS}\ref{l:infty_2_4}) together with \eqref{eq:E_sim_L2}, \eqref{eq:pnorm_D} and \eqref{eq:H_Lp} yields \eqref{eq:Lip_control} for $d=2$.
\end{proof}

Of central importance is the following interpolation estimate controlling the energy in terms of the excess mass and dissipation. We present it here for arbitrary dimension; for the proof of Theorem~\ref{thm:main}, we use it for $d=2$ and $d=1$.
\begin{proposition}[Main interpolation estimate]\label{prop:EED}
Under the condition $\norm{\nabla h}_{\infty}\le 1$, there holds
\begin{align*}
  \E\lesssim \V^\fr{6}{d+5} D^\fr{d+2}{d+5}.
\end{align*}
\end{proposition}
\begin{proof}
This is a consequence of Lemma~\ref{l:V_D}, the interpolation inequality
\begin{align*}
\norm{\nabla h}_2^2\ls \norm{h}_1^{\frac{4}{d+4}}\norm{\nabla^2 h}_2^{\frac{2d+4}{d+4}}
\end{align*}
(cf. Lemma~\ref{l:GNS}\ref{l:EVD}), \eqref{eq:L2_h_H}, and \eqref{eq:H_L2}.
\end{proof}

\subsection{Differential relationships and decay estimates}\label{ss:diff}
We begin with a few elementary results that are true for general nonnegative quantities $\V,\E$, and $D$ satisfying given algebraic and differential relationships. Then starting with Lemma~\ref{l:Ddot} below, we show that the assumed relationships hold true for the Mullins-Sekerka evolution.

Our first observation is an ODE lemma: The gradient flow structure and the interpolation estimate from Proposition \ref{prop:EED} immediately yield decay of $\E$ on any interval $[0,T]$ in terms of the supremum $\V_T$ defined in \eqref{eq:V_T}.

\begin{lemma}\label{l:E_decay}
Suppose that for some $T>0$ the quantities $\V,\E,D:[0,T]\to [0,\infty)$ satisfy
\begin{align}\label{eq:EED}
\ddt \E\le-D,\quad \E\ls \V^{\frac{6}{d+5}}D^{\frac{d+2}{d+5}}.
\end{align}
Then, for all $t\in[0,T]$:
\begin{align}\label{eq:E_decay2}
\E(t)\ls \min\left\{\E_0,\frac{\V_T^2}{t^{\frac{d+2}{3}}}\right\}.
\end{align}
\end{lemma}
\begin{proof}
Combining the relation $\dot \E\le-D$ with the inequality \eqref{eq:EED} yields
\begin{align*}
-\dot \E\gtrsim\V_T^{-\frac{6}{d+2}}\E^{\frac{d+5}{d+2}}.
\end{align*}
An integration in time completes the proof.
\end{proof}
\begin{lemma}\label{l:ED2}
Let $d=2$ or $d=1$. For any $C<\infty$, there exists $\eps_1>0$ with the following property. For any $T>0$, if  $\E,D:[0,T]\to[0,\infty)$ satisfy
\begin{align*}
\ddt \E\le-D,
\end{align*}
and
\begin{align*}
\ddt\D\leq C D^\frac{6-d}{3-d},
\end{align*}
then $\E^{3-d} D^d\le \eps_1$ implies
\begin{align}\label{eq:ED2dot}
\ddt{\bra{\E^{3-d} D^d}}\leq 0.
\end{align}
\end{lemma}
\begin{proof}
This follows from the straightforward computation
\begin{align*}
\ddt{\bra{\E^{3-d} D^d}}&=(3-d)\,\dot{\E}\,\E^{2-d}\,D^d+d\,\dot{D}\,\E^{3-d}\, D^{d-1}\\
&\le -(3-d)\,\E^{2-d}\,D^{d+1}+Cd\,\E^{3-d}\, D^{d-1}\,D^\frac{6-d}{3-d}\\
&\le \E^{2-d}\,D^{d+1}\bra{dC\bra{\E^{3-d}D^d}^\frac{1}{3-d}-(3-d)},
\end{align*}
which has the right sign if $dC\eps_1^{\frac{1}{3-d}}<3-d$. This calculation for $d=1$ appeared already in \cite[Lemma 5.1]{COW}.
\end{proof}

\begin{lemma}\label{l:D_decay}
Let $d=2$ or $d=1$ and $T>0$. Suppose  $\E,D:[0,T]\to [0,\infty)$ satisfy the relations
\begin{align*}
\ddt \E\le-D,\quad \text{and}\quad
\ddt\D\ls D^\frac{6-d}{3-d}\quad
 \text{ on }[0,T].
\end{align*}
Then, for all $t\in [0,T]$ with $t\gtrsim \E_0^{\frac 3d}$ there holds
\begin{align}\label{eq:D_decay2}
D(t)\ls \frac{\E\bra{\frac{t}{2}}}{t}.
\end{align}
\end{lemma}
\begin{proof}
Let $d=2$. By integrating the inequality
\begin{align*}
-\ddt \left(D^{-3}\right)\leq C
\end{align*}
for $C\geq 1$ on the interval $[s,t]$ and multiplying with $D(s)^3D(t)^3$, we obtain the inequality
\begin{align*}
  D(s)^3\geq \frac{D(t)^3}{1+C (t-s)D(t)^3}\geq \frac{D(t)^3}{C+C (t-s) D(t)^3}
\end{align*}
and deduce
\begin{align*}
D(s)\gs \frac{D(t)}{\bra{1+(t-s)D(t)^3}^{\frac 13}}.
\end{align*}
We insert this in
\begin{align*}
\int_{\tau}^t D(s)\dd s\le\E(\tau)-\E(t)\le \E(\tau)
\end{align*}
to obtain
\begin{align*}
\E(\tau)&\gs D(t)\int_{\tau}^t \frac{1}{\bra{1+(t-s)D(t)^3}^{\frac 13}}\dd s=D(t)^{-2}\int_0^{(t-\tau)D(t)^3}\frac{1}{(1+\sigma)^{\frac 13}}\dd \sigma\\
&\gs \min\set{(t-\tau)D(t)(t-\tau),(t-\tau)^{\frac 23}}.
\end{align*}
We finish by choosing $\tau=\frac t2$.

The proof for $d=1$ is similar and is contained in \cite[Proof of (1.13)]{COW}.
\end{proof}

We now check that the assumed differential estimates from Lemmas~\ref{l:ED2} and \ref{l:D_decay} hold true for the Mullins-Sekerka evolution.
The differential inequality for the energy gap is contained in \textbf{(H)}. It remains to
derive a differential inequality for the dissipation.

\begin{lemma}\label{l:Ddot}
Consider a smooth solution of the Mullins-Sekerka evolution with graph structure on $[0,T]$ in $d=2$ or $d=1$ and assume $\norm{\nabla h}_\infty\le 1$.
In $d=1$, there holds
\begin{align}\label{eq:Ddot_1}
\ddt D\lesssim D^{\frac 52}\bra{1+\bra{\E^2D}^{\frac 12}}.
\end{align}
In $d=2$, let $p>2$ be such that the assertion of Lemma~\ref{l:pnorm_D} is satisfied. Then there holds
\begin{align}\label{eq:Ddot}
\ddt D\lesssim D^4\bra{1+\bra{\E D^2}^{\frac{4-p}{2(p-2)}}}.
\end{align}
\end{lemma}
\begin{remark}\label{rem:small_slope}
In view of Corollary \ref{cor:Lip} and Lemmas~\ref{l:ED2} and~\ref{l:Ddot}, the condition $\norm{\nabla h}_\infty \le 1$ will be preserved for the MS evolution in $d=2$ and $d=1$ if $\E^{3-d} D^d$ is small enough initially.
\end{remark}
\begin{proof}[Proof of Lemma~\ref{l:Ddot}]
In $d=1$, the result is contained in \cite[Lemma 4.1 and (4.3)]{COW}.

For $d=2$, we recall the well-known evolution equation for the mean curvature (which can be derived for instance from \cite[(3.8) and (3.10)]{And}):
\begin{align}\label{eq:materialder_H}
  \dot{H}=-\Div_{\tan} \nabla_{\tan} \vel-\abs{\nabla_{\tan} n}^2\vel,
\end{align}
where $\nabla_{\tan}$ and $\Div_{\tan}$ are the surface gradient and the surface divergence, respectively, and $\dot{H}$ represents the change of the curvature in the normal direction.
We compute
\begin{align}
  \ddt D&=\ddt\int_{\R^3}\abs{\babla f}^2\bdx=\ddt\int_{\Omega_+(t)}\abs{\babla f_+}^2\bdx+\ddt\int_{\Omega_-(t)}\abs{\babla f_-}^2\bdx\nonumber\\
  &=\int_{\Omega_+(t)}\ddt\abs{\babla f_+}^2\bdx+\int_{\Omega_-}\ddt\abs{\babla f_-}^2\bdx-\int_{\Gamma}\vel\bra{\abs{\babla f_+}^2-\abs{\babla f_-}^2}\dS\nonumber\\
  \overset{\eqref{eq:f}}&{=}\int_\Gamma -2\dot{f_+}\babla f_+\cdot n
  +2\dot{f_-}\babla f_-\cdot n
  -\vel\bra{\abs{\babla f_+}^2-\abs{\babla f_-}^2}\dS\nonumber\\
  \overset{\eqref{eq:f}}&{=}\int_\Gamma -2\dot{H}\jump{\babla f\cdot n}-\vel\bra{\abs{\babla f_+}^2-\abs{\babla f_-}^2}\dS\nonumber\\
  \overset{\eqref{eq:materialder_H}}&{=}\int_\Gamma 2\bra{\Div_\Gamma \nabla_\Gamma \vel+\abs{\nabla_\Gamma n}^2\vel}V-\vel\bra{\abs{\babla f_+}^2-\abs{\babla f_-}^2}\dS,\label{eq:dtD}
\end{align}
where in the third line we have applied the divergence theorem and $\dot{f_\pm}=\frac{d}{dt}H(h(x,t),t)$ is the (total) time derivative of the curvature of a boundary point. Because $\babla f$ is continuous in the tangential direction, the difference on the right-hand side of \eqref{eq:dtD} satisfies
\begin{align}\label{eq:3rdbinom}
\abs{\babla f_+}^2-\abs{\babla f_-}^2&=\bra{\babla f_+ \cdot n}^2-\bra{\babla f_- \cdot n}^2=[\babla f\cdot n]\bra{\babla f_+ \cdot n+\babla f_- \cdot n}\nonumber\\
\overset{\eqref{eq:f}}&{=}-\vel\bra{\babla f_+ \cdot n+\babla f_- \cdot n}.
\end{align}
Inserting \eqref{eq:3rdbinom} into the right-hand side of \eqref{eq:dtD} and integrating by parts yield
\begin{align}\label{eq:DD_prelim}
 \lefteqn{ \frac{\dd}{\dd t}D+2\int_\Gamma \abs{\nabla_\Gamma \vel}^2\dS}\notag\\
 &\le 2\underbrace{\int_\Gamma  \abs{\nabla_\Gamma n}^2\vel^2\dS}_{=:A} + \underbrace{\int_\Gamma\vel^2\bra{\babla f_+\cdot n+ \babla f_-\cdot n}\dS}_{=:B}.
\end{align}
We will deduce \eqref{eq:Ddot} from \eqref{eq:DD_prelim}, the error estimates
\begin{align}
A&\ls \E^{\frac{4-p}{3p}}D^{\frac{2(3p-4)}{3p}}\norm{\nabla \vel}_{2}^{\frac{2(p+4)}{3p}}\label{eq:E1}\\
  \abs{B}&\ls D^{\frac 23}\norm{\nabla \vel}_2^{\frac 53},\label{eq:E2}
\end{align}
and Young's inequality, absorbing the $\norm{\nabla V}_2^2$ term in the left-hand side, where we use that $\norm{\nabla V}_2^2\ls \int_\Gamma \abs{\nabla_\Gamma \vel}^2\dS$. (Notice that because of $p>2$, one has $2(p+4)/(3p)<2$.)

We start by estimating
\begin{align}\label{eq:first_DD}
  A\ls \norm{\nabla_\Gamma n}_{L^p(\Gamma)}^2\norm{\vel}_{L^{\frac{2p}{p-2}}(\Gamma)}^2.
\end{align}
On the one hand, a simple computation based on \eqref{eq:n} and using $\norm{\nabla h}_\infty\le 1$
reveals
\begin{align}\label{eq:nabla_n_D}
  \norm{\nabla_\Gamma n}_{L^p(\Gamma)}^2\ls \norm{\nabla_\Gamma n}_{p}^2\ls  \norm{\nabla^2 h}_p^2\overset{\eqref{eq:pnorm_D},\eqref{eq:H_Lp}}{\ls} \E^{\frac{4-p}{3p}}D^{\frac{4(p-1)}{3p}}.
\end{align}
On the other hand we estimate
\begin{align}\label{eq:lady}
  \norm{\vel}_{L^{\frac{2p}{p-2}}(\Gamma)}^2\ls \norm{\vel}^2_{\frac{2p}{p-2}}\ls \norm{\vel}_{\Hd^{-1/2}}^{\frac{4(p-2)}{3p}}\norm{\nabla \vel}_{2}^{\frac{2(p+4)}{3p}}.
\end{align}
The last inequality follows from the interpolation inequalities
\begin{align*}
\norm{\vel}_2\ls \norm{\vel}_{\Hd^{-1/2}}^{\frac 23}\norm{\nabla \vel}_2^{\frac 13},
\end{align*}
and
\begin{align*}
\norm{\vel}_q\ls \norm{\vel}_2^{\theta}\norm{\nabla \vel}_2^{1-\theta},
\end{align*}
for $\theta=2/q$, see Lemma~\ref{l:GNS}\ref{l:Lad_q}.

Inserting \eqref{eq:nabla_n_D} and \eqref{eq:lady} into \eqref{eq:first_DD} and using $\norm{\vel}_{\Hd^{-\frac 12}}\ls \norm{\vel}_{\Hd^{-\frac 12}(\Gamma)}$ in combination with Lemma \ref{l:V_D}, we arrive at \eqref{eq:E1}.

We now turn to establishing \eqref{eq:E2} for $B$. The terms with $f_+$ and $f_-$ are handled in the same way, so without loss of generality, we consider  $f_+$ and begin with
\begin{align}\label{eq:second_DD}
  \abs{\int_\Gamma \babla f_+\cdot n\vel^2\dS}\ls \norm{\babla f_+\cdot n}_{\Hd^{-\frac 12}(\Gamma)}\norm{\vel^2}_{\Hd^{\frac 12}(\Gamma)}
  \lesssim D^{\frac 12}\norm{\vel^2}_{\Hd^{\frac 12}(\Gamma)},
\end{align}
where for the second estimate we have applied Lemma \ref{l:V_D}.
For the second term on the right-hand side, we use
\begin{align}\notag
  \norm{\vel^2}_{\Hd^{\frac 12}(\Gamma)}\ls \norm{\vel}_6^{\frac 32}\norm{\nabla \vel}_2^{\frac 12}
\end{align}
(cf.\ Lemma~\ref{l:V^2}) and insert
\eqref{eq:lady} with $p=3$ to obtain
\begin{align}
  \norm{\vel^2}_{\Hd^{\frac 12}(\Gamma)}&\ls \norm{\vel}_{\Hd^{-\frac 12}}^{\frac 13}\norm{\nabla \vel}_2^{\frac{5}3}.\label{eq:H12_est2}
\end{align}
Again applying Lemma \ref{l:V_D} and
inserting the result into \eqref{eq:second_DD} yields \eqref{eq:E2}.

\end{proof}

\subsection{Duality argument}\label{ss:dual}

In this subsection we address the main mathematical challenge of this section, which is to prove within the graph regime that $\V$ remains bounded in terms of its initial data for all time. As in Section~\ref{sec:initial}, the starting point is a dual representation of $\V$. In the graph setting, we use:
\begin{align*}
\V=\sup_{\psi\in L^\infty(\R^d), \norm{\psi}_\infty\le 1}\int \psi h\dx,
\end{align*}
and again we use the solution $\bar u$ of \eqref{eq:uv1}--\eqref{eq:dual_lin2}. 
 To obtain uniform-in-time error estimates for large times, the decay of $\E$ and $D$ will play a central role. Before stating and proving the duality result, we introduce a splitting of the linearization error into kinetic and geometric nonlinearity.
\begin{lemma}[Splitting the error and preprocessing]\label{l:preproc}
Let $\psi\in C^\infty_c(\R^d)$ with $\norm{\psi}_{\infty}\le 1$. In $d=2$ and $d=1$ for a smooth solution of MS with graph structure on $[0,T]$,  and $u$ satisfying \eqref{eq:uv1}--\eqref{eq:dual_lin2}, there holds
\begin{align}\label{integra}
\frac{d}{dt}\int_{\R^d} hu \dx&=\frac{d}{dt}\int_{\R^d}\bar{h}\bar{u}\notag\\
&=\bra{-\int_{\Gamma}\vel\; \ub(x,0)\dS+2\int_{\R^d} f(x,0)\; \partial_z \ub \dx}\\
&+2\int_{\R^d} \bra{f(x,h(x))-f(x,0)}\; \partial_z \ub \dx-2\int_{\R^d}\bra{H(x,h(x))-\Delta  h}\partial_z \ub \dx.
\end{align}
Moreover, defining
\begin{align}
A_4:&=-\int_{\Gamma}\vel\; \ub(x,0)\dS+2\int_{\R^d} f(x,0)\; \partial_z \ub \dx\\
A_5:&=2\int_{\R^d} \bra{f(x,h(x))-f(x,0)}\; \partial_z \ub \dx\\
A_6:&=2\int_{\R^d}\bra{H(x,h(x))-\Delta  h}\partial_z \ub \dx.
\end{align}
the following estimates hold.
\begin{align}
  \abs{A_4}&\ls
  D^{\frac 12}\left(\frac{\V_T^{\frac 1{d+2}}\E^{\frac{d+1}{2(d+2)}}}{(T-t)^{\frac 13}}+\frac{\V_T^{\frac 2{d+2}}\E^{\frac{d}{2(d+2)}}}{(T-t)^{\frac 12}}
  \right),\label{A1initial}\\
 \abs{A_5}&\ls\frac{\V_T^{\frac 12}\D^{\frac 12}}{(T-t)^\frac 13} . \label{eq:A2}\\
\abs{A_6}&\ls   \frac{1}{(T-t)^\frac 23}\min \left \{\norm{\nabla h}_{2}^2, \norm{\nabla h}_{3}^3\right \}.\label{besttwice}
\end{align}
\end{lemma}

\begin{proof} A direct calculation yields
\begin{align}
\frac{d}{dt}\int_{\R^d} \hb\ub \dx&= \int_{\R^d} \ub \partial_t \hb\dx+ \int_{\R^d} \hb \partial_t \ub\dx\nonumber\\
\overset{\eqref{eq:norm_vel},\eqref{eq:uv1}}&{=}\int_{\R^d}\sqrt{1+\abs{\nabla h}^2}\;\vel(x,h(x))\; \ub\dx+2\int_{\R^d} h\; \partial_z\Delta  \ub \dx\nonumber\\
&=\int_{\Gamma}\vel\; \ub(x,0)\dS+2\int_{\R^d} \Delta  h\; \partial_z \ub \dx\nonumber\\
&=\int_{\Gamma}\vel\; \ub(x,0)\dS+2\int_{\R^d} H(x,h(x))\; \partial_z \ub \dx\nonumber\\
&\quad-2\int_{\R^d}\bra{H(x,h(x))-\Delta  h}\partial_z \ub \dx\nonumber\\
&=\int_{\Gamma}\vel\; \ub(x,0)\dS+2\int_{\R^d} f(x,0)\; \partial_z \ub \dx\\
&\quad +2\int_{\R^d} \bra{f(x,h(x))-f(x,0)}\; \partial_z \ub \dx\\
&\quad -2\int_{\R^d}\bra{H(x,h(x))-\Delta  h}\partial_z \ub \dx,
\label{eq:duality_initial}
\end{align}
as desired.

For $A_4$ we compute
\begin{align*}
A_4&=\int_{\Gamma}\vel \ub(x,0)\dS+2\int_{\R^d} f(x,0)\; \partial_z \ub \dx\\
\overset{\eqref{eq:uv2}}&{=}\int_{\Gamma}\vel \ub(x,0)\dS-\int_{\R^{d+1}} \babla f\cdot \babla \ub \bdx\\
\overset{\eqref{eq:f}}&{=}\int_{\Gamma}\vel \ub(x,0)+[\babla f\cdot n] \ub(x,h(x))\dS\\
\overset{\eqref{eq:Vel}}&{=}-\int_{\Gamma}\vel\bra{\ub(x,h(x))-\ub(x,0)}\dS.
\end{align*}
We use $\norm{\nabla h}_{\infty}\leq 1$ and $\norm{\psi}_{\infty}\leq 1$ to estimate
\begin{align}
\abs{A_4}&\ls \norm{\vel}_{\Hd^{-\frac 12}(\Gamma)}\norm{\bra{\ub(\cdot,h(\cdot))-\ub(\cdot,0)}}_{\Hd^{\frac 12}(\Gamma)}\nonumber\\
\overset{\eqref{eq:nablaf_bdry_est}}&{\ls} D^{\frac 12}\norm{\bra{\ub(\cdot,h(\cdot))-\ub(\cdot,0)}}_{L^2(\R^d)}^{\frac 12}\norm{\nabla\bra{\ub(\cdot,h(\cdot))-\ub(\cdot,0)}}_{L^2(\R^d)}^{\frac 12}\notag\\
&\ls
D^{\frac 12}\left(\frac{\V^{\frac 1{d+2}}\E^{\frac{d+1}{2(d+2)}}}{(T-t)^{\frac 13}}+\frac{\V^{\frac 2{d+2}}\E^{\frac{d}{2(d+2)}}}{(T-t)^{\frac 12}}
  \right),
\end{align}
where we have used
\begin{align}
\norm{\bra{\ub(\cdot,h(\cdot))-\ub(\cdot,0)}}_2&=\norm{\int_0^{h(\cdot)}\partial_z \ub(\cdot,z)\dd z}_2\ls \norm{\partial_z \ub}_\infty\norm{h}_2,\notag\\
\text{and}\qquad \norm{\nabla\bra{\ub(\cdot,h(\cdot))-\ub(\cdot,0)}}_2&=\norm{\partial_z \ub(\cdot,h(\cdot))\nabla h(\cdot)+\int_0^{h(\cdot)}\partial_z\nabla \ub(\cdot,z)\dd z}_2\nonumber\\
&\ls \norm{\partial_z \ub}_\infty \norm{\nabla h}_2+\norm{\partial_z\nabla \ub}_{\infty}\norm{h}_2,\notag
\end{align}
together with \eqref{eq:v_decay1}, \eqref{eq:v_decay2}, \eqref{eq:2_1_2}, and \eqref{eq:E_sim_L2}.

We now turn to $A_5$. Starting with
\begin{align*}
A_5&=2\int_{\R^d} \bra{f(x,h(x))-f(x,0)} \partial_z \ub \dx\\
&=2\int_{\R^d} \partial_z \ub \int_0^{h(x)}\partial_z f(x,z)\dz\dx,
\end{align*}
we estimate
\begin{align}
\abs{A_5}&\ls \norm{\partial_z \ub}_\infty \int_{\R^d} \abs{h(x)}^{\frac 12}\bra{\int_0^{h(x)}\abs{\babla f(x,z)}^2\dz}^{\frac 12}\dx\nonumber\\
&\ls \norm{\partial_z \ub}_\infty \bra{\int_{\R^d} \abs{h}\dx}^{\frac 12}\bra{\int_{\R^{d+1}}\abs{\babla f}^2\bdx}^{\frac 12}\nonumber\\
\overset{\eqref{eq:v_decay1}}&{\ls}  \V_T^{\frac 12}\frac{\D^{\frac 12}}{(T-t)^\frac 13} .\notag
\end{align}

Finally, we turn to $A_6$, expressing it in the form
\begin{align}
A_6&=-2\int_{\R^d}\bra{H(x,h(x))-\Delta  h}\partial_z \ub \dx\nonumber\\
\overset{\eqref{eq:H}}&{=}-2\int_{\R^d}\Div\bra{\bra{\frac{1}{\sqrt{1+\abs{\nabla h}^2}}-1}\nabla h}\partial_z \ub \dx\nonumber\\
&=2\int_{\R^d}\bra{\frac{1}{\sqrt{1+\abs{\nabla h}^2}}-1}\nabla h\nabla\partial_z \ub \dx\nonumber\\
&=-2\int_{\R^d}\frac{\nabla h}{\sqrt{1+\abs{\nabla h}^2}}\bra{\sqrt{1+\abs{\nabla h}^2}-1}\nabla\partial_z \ub \dx,\label{eq:B}
\end{align}
which we estimate using
\begin{align*}
\frac{\abs{\nabla h}}{\sqrt{1+\abs{\nabla h}^2}}\le \min\set{1,\abs{\nabla h}},
\end{align*}
to deduce
\begin{align}
\abs{A_6}\ls   \norm{\nabla\partial_z \ub}_\infty\min \left \{\norm{\nabla h}_{2}^2, \norm{\nabla h}_{3}^3\right \}
\overset{\eqref{eq:v_decay2}}\ls \frac{1}{(T-t)^\frac 23}\min \left \{\norm{\nabla h}_{2}^2, \norm{\nabla h}_{3}^3\right \}.\notag
\end{align}

\end{proof}

We are now ready for the duality argument.

\begin{proposition} \label{prop:duality_1}
Let $d=2$ or $d=1$, $T>0$, and consider a smooth solution of the Mullins-Sekerka problem with graph structure on $[0,T]$. Suppose that $\V_T<\infty$, $\norm{\nabla h}_\infty\le 1$ hold for all $t\in [0,T]$ and that $D$ satisfies
\begin{align}\label{eq:D_decay_duality_1}
D(t)\ls \min\left\{\frac{\E_0}{t} , \frac{\V_T^2}{t^\frac{d+5}{3}}\right\}.
\end{align}
Then $\V$ obeys the bound
\begin{align}\label{eq:V_bound_1}
\V_T=\sup_{t\in [0,T]}\V(t)\ls \V_0+\E_0^{\frac {d+1}{d}}.
\end{align}
\end{proposition}
\begin{proof}

Let $u$ satisfy \eqref{eq:uv1}-\eqref{eq:dual_lin2}. We claim that it suffices to establish
\begin{alignat}{2}
&\int_{0}^T \frac{d}{dt}\int_{\R^d} hu\dx\dt
\ls\E_0^\frac 32 +\bra{\V_T^\frac 12 \E_0^\frac 34+\V_T^{\frac 34} \E_0^{\frac 38}}&\qquad&\text{in $d=2$ and}\label{eq:bounds_derivative}\\
&\int_{0}^T \frac{d}{dt}\int_{\R^d} hu\dx\dt
\ls\V_T^\frac 13\E_0^\frac 43 +\V_T^\frac 23 \E_0^\frac 23+\V_T^{\frac 56} \E_0^{\frac 13}&\qquad&\text{in $d=1$}.\label{eq:bounds_derivative_1}
\end{alignat}
Indeed, integrating the left-hand side, taking the supremum over $\psi$, recalling~\eqref{eq:u_decay1}, and applying Young's inequality leads to
\eqref{eq:V_bound_1} for all $T>0$.

We will establish \eqref{eq:bounds_derivative}-\eqref{eq:bounds_derivative_1} based on the error estimates~\eqref{A1initial}--\eqref{besttwice}
in $d=2$ and $d=1$.
The strategy is that we need enough decay for integrability at infinity, but no larger  power than $\V_T^1$ so that we can, as described above, absorb powers of $\V_T$ from the right-hand side into the left-hand side. To this end, we use either the first or second estimate from \eqref{eq:D_decay_duality_1}, as needed. Note that the dissipation decay~\eqref{eq:D_decay_duality_1} is better than the energy decay~\eqref{eq:E_decay} for this purpose, since we get ``more time decay for the same power of $\V_T$.'' Note also that we will make repeated use of Lemma~\ref{l:tint} with $a+b=1$.

For $A_6$, we use \eqref{besttwice} together with
\begin{align*}
\norm{\nabla h}^3_3
&\le \norm{\nabla h}_\infty \norm{\nabla h}^2_2
\overset{\eqref{eq:Lip_control}, \eqref{eq:E_sim_L2}}{\ls} \E^{\frac {9-d}{6}}D^{\frac d6}
\end{align*}
to derive
\begin{align}
&\int_{0}^T\abs{A_6}\dt\ls
\int_{0}^T \frac{\E^\frac {9-d}6 D^\frac {d}{6}}{(T-t)^\frac 23} \dt
\overset{\eqref{eq:D_decay_duality_1}}\ls \E_0^{\frac {9-d}6}\int_{0}^T\frac{1}{(T-t)^{\frac 23}} \left(\frac{\V_T^2}{t^\frac {d+5}3}\right)^\frac {\alpha\,d}{6} \left(\frac{\E_0}{t}\right)^\frac {(1-\alpha)d}{6}\dt.
\end{align}
We choose $\alpha$ to give $t^{\frac 13}$ decay and invoke \eqref{T}, which leads to
\begin{alignat}{3}
&\alpha=0&\qquad&\text{and an error}\quad \E_0^{\frac 32}\qquad&\text{in $d=2$},\\
&\alpha=1&\qquad&\text{and an error}\quad \V_T^\frac 13\E_0^\frac 43\qquad&\text{in $d=1$}.
\end{alignat}

To estimate $A_5$, we use \eqref{eq:D_decay_duality_1} to compute
\begin{align}
	\int_{0}^T \abs{A_5}\dt
\overset{\eqref{eq:A2}}&\ls \V_T^{\frac 12} \int_{0}^T \frac{1}{(T-t)^{\frac 13}}\left(\frac{\V_T^2}{t^\frac {d+5}3}\right)^\frac {\alpha}{2} \left(\frac{\E_0}{t}\right)^\frac {1-\alpha}{2}\qquad \alpha=\frac{1}{d+2}\notag\\
\overset{\eqref{T}}&\ls \V_T^{\frac {d+4}{2(d+2)}} \E_0^{\frac {d+1}{2(d+2)}}.\label{eq:A2_longtime}
\end{align}

Finally we turn to $A_4$. Estimating \eqref{A1initial} as in $A_5$ where for the first time integral we again take $\alpha=\frac{1}{d+2}$ and in the second time integral we instead take $\alpha=0$, we obtain an error
\begin{align*}
\V_T^{\frac{2}{d+2}}  \E_0^{\frac{d+1}{d+2}}.
\end{align*}
\end{proof}

\subsection{Proof of Proposition~\ref{prop:3}}\label{ss:main}

\begin{proof}
\textbf{Step 1: Control of the Lipschitz constant.}
We begin by establishing control on $\E^{3-d} D^d$ and $\norm{\nabla h}_\infty$ for all times. Let $\eps_1$ be the constant from  Lemma~\ref{l:ED2} and $\hat\eps$ be such that $C\hat\eps^{\frac 16}< 1$, where $C<\infty$ is the implicit constant in \eqref{eq:Lip_control}. We set $\eps_2:=\frac{1}{2}\min\{\eps_1,\hat\eps\}$.
 This implies
\begin{align}\label{eq:init_Lip}
 \norm{\nabla h_0}_\infty< 1.
\end{align}
We set
\begin{align*}
  T_1:&=\sup\set{T>0\colon \E^{3-d} D^d\le 2\eps_2 \text{ for all }t\le T}\\
  \text{and}\qquad T_2:&= \sup \set{T>0:\norm{\nabla h}_\infty< 1 \text{ for all }t\le T}
\end{align*}
and note that by smoothness, \eqref{eq:small_quantity}, and \eqref{eq:init_Lip}, there holds $T_1>0$ and $T_2>0$.
Now we define
\[T:=\min\{T_1,T_2\}.\]

On the one hand, we can apply Corollary \ref{cor:Lip} so that $\norm{\nabla h}_\infty< 1$ on $[0,T]$.
On the other hand, $\norm{\nabla h}_\infty\le 1$ implies that we can apply Lemma~\ref{l:Ddot}. Using $\E^{3-d} D^d\leq \eps_1$, we apply Lemma~\ref{l:ED2} to get $\ddt(\E{3-d} D^d)\le 0$ on $[0,T]$. We deduce $T=\infty$.

\textbf{Step 2: Proof of \eqref{st1}, \eqref{st2}, \eqref{st3}.} Next we define
\begin{align*}
T_3\coloneqq \sup \set{T>0:\V\le \tilde C(\V_0+\E_0^{\frac{d+1}{d}}) \text{ for all }t\le T},
\end{align*}
where $\tilde C>1$ is a universal constant to be specified below and smoothness implies
$T_3>0$.

Our goal is to prove  $T_3=\infty$ for $\tilde C$ large enough.

By Proposition \ref{prop:EED} and Lemma \ref{l:E_decay} we see that
\begin{align*}
\E\ls \frac{\V^2_{T_3}}{t^{\frac{d+2}3}} \qquad \text{for}\quad t\in[0,T_3].
\end{align*}
Combining this with Lemma \ref{l:Ddot}, $\E^{3-d} D^d<2\eps_2$,  and Lemma \ref{l:D_decay}, we obtain
\begin{align*}
D\ls \frac{\E\bra{\frac t2}}{t}\ls \frac{\V^2_{T_3}}{t^{\frac{d+5}{3}}}\qquad\text{for all }t\in [t_\ast,T_3],
\end{align*}
where $t_\ast\sim \E_0^\frac{3}{d}$ is the timescale from Lemma \ref{l:D_decay}. Consequently Proposition \ref{prop:duality_1} yields that
\begin{align*}
\V_{T_3}\leq C \bra{\V_0+\E_0^{\frac{d+1}{d}}}
\end{align*}
for a universal constant $C<\infty$. Choosing $\tilde C>C$ implies $T_3=\infty$.

\textbf{Step 3: Proof of \eqref{st4}.} A combination of Corollary~\ref{cor:Lip} and Step 2 delivers the bound on $\norm{\nabla h}_\infty$ in \eqref{st4} for $t\geq t_\ast$. The bound on $\norm{h}_\infty$ follows from item~\ref{l:h_infty} of Lemma~\ref{l:GNS}, Lemma~\ref{l:V_D}, and the bound on $\norm{\nabla h}_\infty$.

\end{proof}
\section*{Acknowledgements}
We are indebted to Jonas Hirsch for his kind and insightful explanations of the results of Allard.
We gratefully acknowledge Sebastian Hensel for pointing out an error in our first version and the two anonymous referees for valuable corrections and suggestions, which helped to improve the final version of the article. In addition we thank the Max Planck Institute for Mathematics in the Sciences, which was fundamental in supporting our collaboration. RS also gratefully recognizes partial funding from the DFG (grant WE 5760/1-1). Finally, we thank the Hausdorff Research Institute for Mathematics in Bonn and the trimester program on Evolution of Interfaces, where early discussions helped to launch this exploration.

\appendix

\section{Elementary bounds}
For completeness, we collect here a few elementary results that we make use of in the paper.

\begin{lemma}\label{l:decaybaru}
Let $\psi\in C^\infty_c(\R^d)$. For the function $\bar{u}$ satisfying \eqref{eq:uv1}--\eqref{eq:dual_lin2}, there holds
\begin{align}
  \abs{\bar{u}(\bx)}\lesssim \frac{1}{\abs{\bx}^d} \quad\text{for }\;\abs{\bx}\gg 1.\label{decaybaruapp}
  \end{align}
\end{lemma}
\begin{proof}
To establish \eqref{decaybaruapp}, we observe that because of the decay rate of $G$ from \eqref{Gcay} and $\psi\in C^\infty_c(\R^d)$, there holds
\begin{align}
\norm{u(t)}_{L^1(\R^d)}\lesssim 1\qquad\text{and}\qquad \abs{u(x,t)}\lesssim T^{1/3}\abs{x}^{-(d+1)}\;\;\text{for large $x$ and all $t<T$}.\label{startu}
\end{align}
For $\abs{z}>\abs{x}$, \eqref{decaybaruapp} directly follows from the decay of the Poisson kernel $\abs{P(x,z)}\ls \abs{z}^{-d}$ and $\norm{u}_{L^1(\R^d)}\lesssim 1$. For  $\abs{x}\geq\abs{z}$ and $\abs{x}$ large enough, we write
 \begin{align*}
  \abs{\bar{u}(x,z)}&=\abs{\int_{\abs{x-y}\ge \frac 12 \abs{x}}P(x-y,z)u(y)\dd y+\int_{\abs{x-y}\le \frac 12 \abs{x}}P(x-y,z)u(y)\dd y}\\
  &\ls \sup_{\abs{x-y}\ge \frac 12\abs{x}}\frac{\abs{z}}{{\abs{x-y}^{d+1}}}\norm{u}_1+\norm{P(\cdot,z)}_1\sup_{\abs{x-y}\le \frac 12\abs{x}}\abs{u(y)}\\
  &\overset{\eqref{startu}}\ls \frac{\abs{z}}{\abs{x}^{(d+1)}}+\frac{1}{\abs{x}^{d}}\ls \frac{1}{\abs{\bx}^{d}}.
\end{align*}
\end{proof}

\begin{lemma}\label{l:linear_decay}
Let $\psi\in L^\infty(\R^d)$ and let $\bar u$ satisfy~\eqref{eq:uv1}--\eqref{eq:dual_lin2}.
Then we have the following estimates in terms of the terminal data:
\begin{alignat}{2}
\norm{\bar u}_\infty&\ls\norm{ u}_\infty&&\ls \norm{\psi}_\infty,\label{eq:u_decay1}\\
\norm{\nabla \bar u}_\infty&\ls\norm{\nabla u}_\infty&&\ls (T-t)^{-\fr13}\norm{\psi}_\infty,\label{eq:u_decay2}\\
\norm{\nabla^2 \bar u}_\infty&\ls\norm{\nabla^2  u}_\infty&&\ls (T-t)^{-\frac 23}\norm{\psi}_\infty,\label{eq:u_decay3}\\
\norm{\nabla^3 \bar u}_\infty&\ls\norm{\nabla^3  u}_\infty&&\ls (T-t)^{-1}\norm{\psi}_\infty.\label{eq:u_decay4}
\end{alignat}
Furthermore recalling $v=-|\nabla|u=\partial_z \bar u(x,0)$, we have
\begin{alignat}{2}
\norm{\partial_z \bar u}_\infty&\ls\norm{v}_\infty&&\ls (T-t)^{-\fr13}\norm{\psi}_\infty,\label{eq:v_decay1}\\
\norm{\partial_z\nabla \bar u}_\infty&\ls\norm{\nabla v}_\infty&&\ls (T-t)^{-\frac 23}\norm{\psi}_\infty,\label{eq:v_decay2}\\
\norm{\partial_z^2\nabla \bar u}_\infty&\ls\norm{\partial_z\nabla v}_\infty&&\ls (T-t)^{-1}\norm{\psi}_\infty.\label{eq:v_decay3}
\end{alignat}
\end{lemma}
\begin{proof}
Recall that $\bar u$ is the harmonic extension of the function $u$ satisfying
\begin{align}
\partial_t u+2\abs{\nabla}\Delta u&=0 \quad \text{in }[0,T)\times \R^d,\\
u(T)&=\psi \quad \text{in } \R^d.
\end{align}
By the maximum principle all estimates for $u,v$ and their derivatives carry over to $\bar u$ and $\partial_z \bar u$. The estimates for $u$ are a consequence of
\begin{align}
\norm{\nabla^m u(t)}_\infty\le \norm{\nabla^m G(T-t)}_1\norm{\psi}_\infty, \quad m=0,1,2,
\end{align}
where $G$ is the kernel from \eqref{eq:profile}.
Bounds for $\norm{\nabla^m G}_1$ are derived by combining $L^\infty$ bounds and decay at infinity coming from the Fourier transform.

The bounds for $v$ are slightly more involved.
We use that $\partial_z P(x,0)= C \Delta \Phi(x)$, for some constant $C$ and where $\Phi$ is the fundamental solution of the Laplace equation in $d+1$ dimensions. Letting $\varphi$ be a smooth cut-off function supported on $\R^d\setminus B_{R}(0)$ with $\varphi\equiv 1$ in $\R^d\setminus B_{2R}(0)$ and $\norm{ \nabla \varphi}\ls R^{-1}$, we compute
\begin{align*}
	\frac 1C v(x)=&\int_{\R^d}\Delta \Phi(y)u(x-y)\dd y\\
	=&\int_{\R^d}(1-\varphi(y))\Delta \Phi(y)u(x-y)\dd y-\int_{\R^d}\nabla\varphi(y)\nabla \Phi(y)u(x-y)\dd y\\
	&+\int_{\R^d}\varphi(y)\nabla \Phi(y)\nabla u(x-y)\dd y\\
	=&\int_{\R^d}(1-\varphi(y))\Delta \Phi(y)u(x-y)\dd y-\int_{\R^d}\nabla\varphi(y)\nabla \Phi(y)u(x-y)\dd y\\
	&-\int_{\R^d}\nabla \varphi(y))\Phi(y)\nabla u(x-y)\dd y+\int_{\R^d}\varphi(y))\Phi(y)\Delta u(x-y)\dd y
\end{align*}
Using the decay of $\Phi$ and its derivatives, and the support of $\varphi, \nabla \varphi$, we obtain for $d=2$ the estimate
 \begin{align*}
	\abs{v(x)}\ls R^{-1}\norm{u}_\infty +\norm{\nabla u}_\infty+R\norm{\nabla^2 u}_\infty.
\end{align*}
The choice $R=(T-t)^{\frac 13}$ yields \eqref{eq:v_decay1} from \eqref{eq:u_decay1}, \eqref{eq:u_decay2}, \eqref{eq:u_decay3} for $d=2$. In $d=1$ we use a Taylor expansion of $u$ on the support of $\varphi$ to extinguish the logarithmic singularity. The higher derivative estimates \eqref{eq:v_decay2}-\eqref{eq:v_decay3} work analogously.

\end{proof}

In the duality proof for large times, we make repeated use of the following fact.
\begin{lemma}\label{l:tint}
  For constants $0<a,b<1$ such that $a+b\geq 1$, there holds
  \begin{align}
    \int_0^T \frac{1}{(T-t)^a}\frac{1}{t^b}\,\dt\lesssim \frac{1}{T^{a+b-1}}.\label{T}
  \end{align}
\end{lemma}
\begin{proof}
  The proof can be obtained by elementary integration after separating the region of integration into $[0,T/2]$ and $[T/2,T]$ and noting that on the first region, the first term in the integrand is bounded by $\lesssim T^{-a}$ and analogously on the second region.
\end{proof}

The following interpolation result is slightly involved, and for the convenience of the reader we include the proof.
\begin{lemma}\label{l:V^2}
Let $d\in \N$ and $V\in L^6(\R^d)\cap \Hd^1(\R^d)$. Then, $V^2\in \Hd^{\frac 12}(\R^d)$ and
\begin{align}\label{eq:V^2}
\norm{V^2}^2_{\Hd^{\frac 12}}\ls \norm{V}_6^3\norm{\nabla V}_2.
\end{align}
\end{lemma}
\begin{proof}
We start with
\begin{align*}
\norm{V^2}_{\Hd^{\frac 12}}^2&=\int_{\R^d} \abs{k}\bra{\widehat{V^2}}^2(k)\dd k=\frac 1{(2\pi)^{\frac d2}}\int_{\R^d} \widehat{V^2}(k)\frac k{\abs{k}}\int_{\R^d}ke^{-ikx}V^2(x)\dd x\;\dd k\\
&=-2\frac 1{(2\pi)^{\frac d2}}\int_{\R^d} \widehat{V^2}(k)\frac{ik}{\abs{k}}\int_{\R^d}e^{-ikx}V(x)\nabla V(x)\dd x\;\dd k\\
&=-2\int_{\R^d} \widehat{V^2}(k)\frac{ik}{\abs{k}}\bra{\hat{V}\ast\widehat{\nabla V}}(k)\dd k\\
&=-\int_{\R^d}\int_{\R^d} \frac{ik}{\abs{k}}\widehat{V^2}(k)\hat{V}(k-k^\prime)\widehat{\nabla V}(k^\prime)\dd k^\prime\;\dd k\\
&=\int_{\R^d}\widehat{\nabla V}(k^\prime)\int_{\R^d} \frac{-ik}{\abs{k}}\widehat{V^2}(k)\hat{V}(k-k^\prime)\dd k\;\dd k^\prime.
\end{align*}
By the Plancherel theorem, \eqref{eq:V^2} follows from the estimate
\begin{align}\label{eq:V^3}
\norm{G}_2^2\coloneqq\norm{\F^{-1}\bra{\int_{\R^d} \frac{-ik}{\abs{k}}\widehat{V^2}(k)\hat{V}(k-k^\prime)\dd k}}_2^2\ls \norm{V}_6^6,
\end{align}
where $\F^{-1}$ is the inverse Fourier transform.
To prove \eqref{eq:V^3} we write
\begin{align*}
G=\int_{\R^d}A[V^2](x)V(-x),
\end{align*}
where $A$ is the operator associated to the Fourier-multiplier $\frac{-ik}{\abs{k}}$.
This implies
\begin{align*}
\norm{G}_2^2=\int_{\R^d}A[V^2]^2(x)V^2(x)\dd x\le \norm{V}_6^2\norm{A[V^2]}_3^2,
\end{align*}
from which \eqref{eq:V^3} follows by an application of the H\"ormander-Mikhlin theorem ($\norm{A[V^2]}_3\ls \norm{V^2}_3$)
\end{proof}

We collect here without proof all interpolation estimates that we use and that are more or less direct consequences of the Gagliardo-Nirenberg-Sobolev inequalities.

\begin{lemma}\label{l:GNS}
The following interpolation inequalities hold.
\begin{enumerate}[label=(\roman*)]
\item\label{l:infty_2_4} Let $g\in L^2(\R^2)$ with $\nabla g\in L^p(\R^2)$ for some $p>2$. Then $g\in L^\infty(\R^2)$ and
\begin{align}\label{eq:infty_2_4}
\norm{g}_\infty\ls \norm{g}_2^{\frac{p-2}{2(p-1)}}\norm{\nabla g}_p^{\frac{p}{2(p-1)}}
\end{align}
For $h\in \dot{H}^1(\R^2)$ with $\nabla^2h\in L^p(\R^2)$ this entails in particular
\begin{align}\label{eq:nabla_h_E}
\norm{\nabla h}_\infty\ls \norm{\nabla h}_2^{\frac{p-2}{2(p-1)}}\norm{\nabla^2h}_p^{\frac{p}{2(p-1)}}.
\end{align}
\item\label{l:2_1_2} Let $h\in L^1(\R^d)$ with $\nabla h\in L^2(\R^d)$. Then $h\in L^2(\R^d)$ and
\begin{align}\label{eq:2_1_2}
\norm{h}_2\ls \norm{h}_1^{\frac{2}{d+2}}\norm{\nabla h}_2^{\frac{d}{d+2}}.
\end{align}
\item\label{l:EVD} Let $h\in L^1(\R^d)$ with $\nabla^2 h\in L^2(\R^d)$. Then $\nabla h\in L^2(\R^d)$ and
\begin{align}\label{eq:EVD}
\norm{\nabla h}_2\ls \norm{h}_1^{\frac{2}{d+4}}\norm{\nabla^2 h}_2^{\frac{d+2}{d+4}}.
\end{align}
\item\label{l:h_infty}
Let $h\in L^1(\R^2)$ with $\nabla h\in L^\infty(\R^2)\cap L^2(\R^2)$. Then $h\in L^\infty(\R^2)$ and
\begin{align}\label{eq:h_infty}
  \norm{h}_\infty\ls \norm{h}_1^{\frac 15}\norm{\nabla h}_2^{\frac 25}\norm{\nabla h}_\infty^{\frac 25}.
\end{align}
If $h\in L^1(\R)$ with $h_x\in L^2(\R)$ then $h\in L^\infty(\R)$  and
\begin{align}\label{eq:h_infty_1}
\norm{h}_\infty \ls \norm{h}_1^\frac 13\norm{h_x}_2^\frac 23.
\end{align}
\item\label{l:Lad_q} Let $q\in (2,\infty)$. If $g\in L^2(\R^2)$ with $\nabla g\in L^2(\R^2)$, then $g\in L^q(\R^2)$ and
\begin{align}\label{eq:Lad_q}
\norm{g}_q\ls \norm{g}_2^{\frac 2q}\norm{\nabla g}_2^{\frac{q-2}{q}}.
\end{align}
\end{enumerate}
\end{lemma}

\end{document}